\documentclass[reqno, 10pt]{amsart}
\usepackage{amssymb}
\usepackage{amsmath}
\usepackage[usenames, dvipsnames]{color}
\usepackage{esint}
\usepackage{hyperref}
\usepackage{todonotes}
\usepackage{verbatim}
\usepackage{mathrsfs}
\usepackage{cite}
\usepackage[nobysame]{amsrefs}
\usepackage{bbm}
\def\MR#1{}

\theoremstyle{plain}
\newtheorem{theorem}{Theorem}[section]
\newtheorem{lemma}[theorem]{Lemma}
\newtheorem{corollary}[theorem]{Corollary}

\theoremstyle{definition}

\theoremstyle{remark}
\newtheorem{remark}[theorem]{Remark}

\newcommand{\dv}{\operatorname{div}}
\newcommand{\osc}{\operatorname{osc}}

\numberwithin{equation}{section}

\newcommand{\bR}{\mathbb{R}}

\newcommand{\bS}{\mathbb{S}}

\newcommand\cD{\mathcal{D}}

\makeatletter
\def\dashint{\operatorname%
{\,\,\text{\bf--}\kern-.98em\DOTSI\intop\ilimits@\!\!}}
\makeatother

\begin{document}
\title[The insulated conductivity problem]{The insulated conductivity problem with $p$-Laplacian}
\author[H. Dong]{Hongjie Dong}

\author[Z. Yang]{Zhuolun Yang}

\author[H. Zhu]{Hanye Zhu}

\address[H. Dong]{Division of Applied Mathematics, Brown University, 182 George Street, Providence, RI 02912, USA}
\email{Hongjie\_Dong@brown.edu}

\address[Z. Yang]{Institute for Computational and Experimental Research in Mathematics, Brown University, 121 South Main Street, Providence, RI 02903, USA}
\email{zhuolun\_yang@brown.edu}

\address[H. Zhu]{Division of Applied Mathematics, Brown University, 182 George Street, Providence, RI 02912, USA}
\email{hanye\_zhu@brown.edu}

\thanks{H. Dong is partially supported by Simons Fellows Award 007638 and the NSF under agreement DMS-2055244.}
\thanks{Z. Yang is partially supported by Simons Foundation Institute Grant Award ID 507536.}
\thanks{H. Zhu is partially supported by the NSF under agreement DMS-2055244.}
%
%
\begin{abstract}
We study the insulated conductivity problem with closely spaced insulators embedded in a homogeneous matrix where the current-electric field relation is the power law $J = |E|^{p-2}E$. The gradient of solutions may blow up as $\varepsilon$, the distance between insulators, approaches to 0. We prove an upper bound of the gradient to be of order $\varepsilon^{-\alpha}$, where $\alpha = 1/2$ when $p \in(1,n+1]$ and any $\alpha > n/(2(p-1))$ when $p > n + 1$. We provide examples to show that this exponent is almost optimal  in 2D. Additionally in dimensions $n \ge 3$, for any $p > 1$, we prove another upper bound of order $\varepsilon^{-1/2 + \beta}$ for some $\beta > 0$, and show that $\beta \nearrow 1/2$ as $n \to \infty$.
\end{abstract}

\maketitle

\section{Introduction and Main results}
We investigate the phenomenon of electric field concentration in high-contrast composites. Such a phenomenon can occur when two approaching inclusions possess material properties that differ significantly from the background matrix (see e.g. \cites{Kel, Mar, BudCar}). The study of this area originated from \cite{BASL}, where the problem with inclusions closely located in a linear background medium was studied numerically. In this paper, we study the scenario in which the inclusions are insulators, and the background matrix follows the current-electric field relation described by the power law:
\begin{equation}\label{power_law}
J = \sigma |E|^{p-2} E, \quad p > 1,
\end{equation}
where $J$, $E$, and $\sigma$ denote current, electric field, and conductivity, respectively. Physically, such power law can occur in various materials, including dielectrics, plastic moulding, plasticity phenomena, viscous flows in glaciology, electro-rheological and thermo-rheological fluids. See the second paragraph of \cite{BIK} and the references therein.

Let us describe the mathematical setup: let $n\ge 2$, $\Omega \subset \mathbb{R}^n $ be a bounded domain with $C^{1,1}$ boundary containing two $C^{1,1}$ open sets $\mathcal{D}_{1}$ and $\mathcal{D}_{2}$ with dist$(\mathcal{D}_1 \cup \mathcal{D}_2, \partial \Omega) > c > 0$. Let
$$\varepsilon: = \mbox{dist}(\mathcal{D}_1, \mathcal{D}_2),$$
$\widetilde{\Omega} := \Omega \setminus \overline{(\mathcal{D}_1 \cup \mathcal{D}_2)}$, and $\sigma = \mathbbm{1}_{\widetilde{\Omega}}$. The voltage potential $u$ satisfies the following $p$-Laplace equation with $p >1$:
\begin{equation}\label{equzero}
\left\{
\begin{aligned}
-\dv (|D u|^{p-2} D u) &=0 \quad \mbox{in }\widetilde{\Omega},\\
\frac{\partial u}{\partial \nu} &= 0 \quad \mbox{on}~\partial\mathcal{D}_{i},~i=1,2,\\
 u &= \varphi \quad \mbox{on } \partial \Omega,
\end{aligned}
\right.
\end{equation}
where $\varphi\in{C}^{1,1}(\partial\Omega)$ is given, and $\nu = (\nu_1, \ldots, \nu_n)$ denotes the inner normal vector on $\partial\mathcal{D}_{1} \cup \partial\mathcal{D}_{2}$.

Our goal is to quantitatively analyze the concentration of the electric field $E = -D u$ between the inclusions, and this is a challenging problem even in the linear case when $p=2$. While the optimal blow-up rate for the linear case in two dimensions was captured about two decades ago in \cites{AKL,AKLLL}, the optimal rate in dimensions $n \ge 3$ was only recently identified in \cites{DLY,DLY2}. This optimal rate is linked to the first non-zero eigenvalue of an elliptic operator on $\bS^{n-2}$, which is determined by the principal curvatures of the inclusions. This phenomenon is completely different from the perfect conductivity problem, where the optimal blow-up rates do not depend on the curvatures of the inclusions (see \cites{AKL,AKLLL,BLY1}). For other earlier work on the linear insulated conductivity problem, we refer the reader to \cites{BLY2, LY, LY2, Y3, We}.
In the case when the current-electric field relation is given by \eqref{power_law}, Gorb and Novikov in \cite{GorNov} and Ciraolo and Sciammetta in \cite{CirSci} studied the field concentration when $\mathcal{D}_1$ and $\mathcal{D}_2$ are perfect conductors. They proved that for $n \ge 2$,
$$
\| \nabla u\|_{L^\infty(\widetilde \Omega)} \le
\left\{
\begin{aligned}
&C\varepsilon^{-\frac{n-1}{2(p-1)}} && p > \frac{n+1}{2},\\
&C\varepsilon^{-1}|\log \varepsilon|^{\frac{1}{p-1}} && p = \frac{n+1}{2},\\
&C\varepsilon^{-1} && p < \frac{n+1}{2}.
\end{aligned}
\right.
$$
These bounds were shown to be optimal in their respective papers. However, the phenomenon of electric field concentration between insulators has not been studied before.

Before stating our main results, let us introduce some notation. We denote $x = (x', x_n)$, where $x' \in \mathbb{R}^{n-1}$. By choosing a coordinate system properly, we can assume that near the origin,  the part of $\partial \mathcal{D}_1$ and $\partial \mathcal{D}_2$, denoted by $\Gamma_+$ and $\Gamma_-$, are respectively the graphs of two $C^{1,1}$ functions in terms of $x'$. That is
\begin{align*}
\Gamma_+ = \left\{ x_n = \frac{\varepsilon}{2}+h_1(x'),~|x'|<1\right\},\quad \Gamma_- = \left\{ x_n = -\frac{\varepsilon}{2}+h_2(x'),~|x'|<1\right\},
\end{align*}
where $h_1$ and $h_2$ are $C^{1,1}$ functions satisfying
\begin{equation}\label{fg_0}
h_1(0')=h_2(0')=0,\quad D_{x'}h_1(0')=D_{x'}h_2(0')=0,
\end{equation}
\begin{equation}\label{fg_1}
c_1 |x'|^2 \le h_1(x')-h_2(x')\quad\mbox{for}~~0<|x'|<1,
\end{equation}
\begin{equation}\label{def:c_2}
    \|h_1\|_{C^{1,1}}\leq c_2, \quad\|h_2\|_{C^{1,1}}\leq c_2,
\end{equation}
with some positive constants $c_1$, $c_2$. For $x_0 \in \widetilde\Omega$, $0 < r\leq 1$, we denote
\begin{align*}
\Omega_{x_0,r}:=\left\{(x',x_{n})\in \widetilde{\Omega}:~-\frac{\varepsilon}{2}+h_2(x')<x_{n}<\frac{\varepsilon}{2}+h_1(x'),~|x'-x_0' |<r\right\},
\end{align*}
and $\Omega_{r}:=\Omega_{0,r}$.
We use $B_r(x_0)$ to denote the open ball of radius $r$ centered at $x_0$ and we set
$$
B_r=B_r(0),\quad
\Omega_r(x_0)=\Omega\cap B_r(x_0).
$$
By classical $C^{1,\alpha}$ estimates for the $p$-Laplace equation and the maximum principle (see e.g. \cites{Lie,Vaz}), the solution $u \in W^{1,p}(\widetilde\Omega)$ of \eqref{equzero} satisfies
\begin{equation*}
\|u\|_{L^\infty(\widetilde\Omega)} + \|u\|_{C^{1,\alpha}(\widetilde\Omega \setminus \Omega_{1/2})} \le C\|\varphi\|_{C^{1,1}(\partial \Omega)}
\end{equation*}
for some constants $\alpha=\alpha(n,p)\in (0,1)$ and $C>0$ independent of $\varepsilon$, $\varphi$, and $u$.
As such, we focus on the following problem near the origin:
\begin{equation}\label{main_problem_narrow}
\left\{
\begin{aligned}
-\dv (|D u|^{p-2} D u) &=0 \quad \mbox{in }\Omega_{1},\\
\frac{\partial u}{\partial \nu} &= 0 \quad \mbox{on } \Gamma_+ \cup \Gamma_-.\\
\end{aligned}
\right.
\end{equation}

For any domain $\mathcal{D}$, we denote the oscillation of $u$ in $\mathcal{D}$ by
$$
\underset{\mathcal{D}}{\osc}~u:= \sup_\mathcal{D} u - \inf_\mathcal{D} u.
$$

Our first main result is the following pointwise gradient estimate of order $\varepsilon^{-1/2}$ for any $p>1$ and $n\geq 2$.
\begin{theorem}
\label{thm-1/2}
Let $h_1$, $h_2$ be  $C^{1,1}$ functions satisfying \eqref{fg_0}-\eqref{def:c_2}, $p>1$, $n \ge 2$, $\varepsilon \in(0,1)$, and $u \in W^{1,p}({\Omega}_1)$ be a solution of \eqref{main_problem_narrow}. Then there exists a constant $C>0$ depending only on $n$, $p$, $c_1$, and $c_2$, such that for any $x\in \Omega_{1/2}$ and $\eta=\frac{1}{4}(\varepsilon+|x'|^2)^{\frac{1}{2}}$,
\begin{equation}\label{gradient-1/2}
|D u(x)| \le  C (\varepsilon+|x'|^2)^{-\frac{1}{2}}\underset{\Omega_{x,\eta}}{\osc}~u.
\end{equation}
\end{theorem}
When $n\geq 3$, we improve the upper bound in Theorem \ref{thm-1/2} to the order of $\varepsilon^{-1/2 + \beta}$ for some $\beta>0$.
\begin{theorem}
\label{Theorem_improved}
Let $h_1$, $h_2$ be $C^{1,1}$ functions satisfying \eqref{fg_0}-\eqref{def:c_2}, $p > 1$, $n \ge 3$, $\varepsilon\in(0,1)$, and $u \in W^{1,p}(\Omega_{1})$ be a solution of \eqref{main_problem_narrow}. Then there exist positive constants $C$ and $\beta$ depending only on $n$, $p$, $c_1$, and $c_2$, such that for any $x\in\Omega_{1/2}$,
\begin{equation}\label{gradient_-1/2+beta}
|D u(x)| \le C(\varepsilon + |x'|^2)^{-1/2 + \beta}\underset{\Omega_{1}}{\osc}~u.
\end{equation}
\end{theorem}
When $p > n + 1$, we derive a more explicit upper bound, under an additional assumption that $h_1$ and $h_2$ are $C^{2}$ strictly convex and strictly concave, respectively. That is, for some positive constants $\kappa_1$ and $\kappa_2$,
\begin{equation}\label{fg_convex}
\kappa_1I_{n-1} \le D^2 h_1(x') \le \kappa_2 I_{n-1}, \quad  \kappa_1I_{n-1} \le -D^2 h_2(x') \le \kappa_2 I_{n-1} \quad\mbox{for}~~0\leq |x'|<1.
\end{equation}
Our pointwise gradient estimate of order ${-\frac{n+2\delta}{2(p-1)}}$, for any $\delta>0$ when $p > n + 1$, is as follows.

\begin{theorem}\label{thm:2d}
Let $n\geq 2$, $p > n+1$, $h_1$, $h_2$ be $C^{2}$ functions  satisfying \eqref{fg_0} and \eqref{fg_convex}, and let $u \in W^{1,p}(\Omega_1)$ be a solution of \eqref{main_problem_narrow}. Then for any $\delta>0$ and $\varepsilon\in(0,1)$, we have
\begin{equation} \label{gradient_2d}
|D u(x)| \le C (\varepsilon + x_1^2)^{-\frac{n+2\delta}{2(p-1)}}\underset{\Omega_{1}}{\osc}~u \quad \mbox{for}~~x \in \Omega_{1/2},
\end{equation}
where $C$ is a positive constant depending on $n$, $p$, $\delta$, $\kappa_1$, $\kappa_2$, and the modulus of continuity for $D^2h_1(x')$ and $D^2h_2(x')$ at $x' = 0$.
\end{theorem}

Furthermore,  we show that when $n=2$, the blow-up exponents $-1/2$ for $p\leq 3$ and $-1/(p-1)$ for $p>3$ obtained in Theorems \ref{thm-1/2} and \ref{thm:2d} are critical in the sense that
\begin{theorem}\label{thm:2d:2}
For $n =2$, $ p>1$, $\varepsilon\in(0,1)$, let $\Omega = B_5$, and $\cD_1,\cD_2$ be the unit balls center at $(0,1+\varepsilon/2)$ and $(0,-1-\varepsilon/2)$, respectively. Let $\varphi = x_1$ and $u \in W^{1,p}(\widetilde \Omega)$ be the solution of \eqref{equzero}. Then for any $\delta> 0$, there exists a positive constant $C$ depending only on $p$ and $\delta$, such that when $ p \in(1, 3]$,
$$
\| D u \|_{L^\infty(\widetilde \Omega \cap B_{8\sqrt{\varepsilon/\delta}})} \ge \frac{1}{C} \varepsilon^{\frac{-1 + \delta}{2}},
$$
and when $p > 3$,
$$
\| D u \|_{L^\infty(\widetilde \Omega \cap B_{8\sqrt{\varepsilon/\delta}})} \ge \frac{1}{C} \varepsilon^{\frac{-1 + \delta}{p-1}}.
$$
\end{theorem}

Finally, we also establish a blow-up rate of $\varepsilon^{-1/2+\beta}$ for the gradient, for any $p>1$ and sufficiently large $n$, with more explicit constant $\beta\in[0,1/2)$. For this, we impose a further assumption that $h_1$ and $h_2$ are $C^{2,\text{Dini}}$ strictly convex and strictly concave respectively, satisfying \eqref{fg_convex} for some positive constants $\kappa_1$ and $\kappa_2$.

\begin{theorem}\label{thm:bern}
Let $h_1$, $h_2$ be $C^{2,\text{Dini}}$ functions satisfying \eqref{fg_0} and \eqref{fg_convex}, $p > 1$, $ \beta\in[0,1/2)$, $ \varepsilon \in(0,1)$, and $u \in W^{1,p}(\Omega_{1})$ be a solution of \eqref{main_problem_narrow}. If $n$, $p$, and $\beta$ satisfy either
\begin{equation}\label{np_relation_1}
p \ge 2, \quad n \ge \frac{5(p-1)}{2} \left( \frac{p+1 - 2\beta(p-1)}{2} + \frac{ \kappa_2}{(1-2\beta)\kappa_1}  \right) + 1,
\end{equation}
or
\begin{equation}\label{np_relation_2}
1 < p < 2, \quad n \ge \frac{5}{2} \left( \frac{3 - 2\beta}{2} + \frac{ \kappa_2}{(1-2\beta)\kappa_1}  \right) + 3-p,
\end{equation}
then there exist a positive constant $C$ depending only on $n$, $p$, $\beta$, $\kappa_1$, and $\kappa_2$, such that
\begin{equation}\label{gradient_-1/2_improved}
|D u(x)| \le C \|u\|_{L^\infty(\Omega_{1})} (\varepsilon + |x'|^2)^{-\frac{1}{2}+\beta } \quad \mbox{for}~~x \in \Omega_{1/2}.
\end{equation}
\end{theorem}
\begin{remark}
By \eqref{np_relation_1} and \eqref{np_relation_2}, when $n \to \infty$, $\beta$ can be chosen arbitrarily close to $1/2$. In view of \eqref{gradient_-1/2_improved}, the singularity of $Du$ diminishes as the dimension $n$ increases. We also note that by refining the inequalities in the proof of Theorem \ref{thm:bern} further, it is possible to improve the lower bounds of $n$ in both \eqref{np_relation_1} and \eqref{np_relation_2}. However, we have decided not to pursue this in the current paper.
\end{remark}

The rest of the paper is organized as follows. In the next section, we give the proof of Theorem \ref{thm-1/2} using mean oscillation estimates. In Section \ref{sec3}, we demonstrate the proof of Theorem \ref{Theorem_improved} by utilizing a delicate change of variables, an extension argument, and the Krylov-Safonov Harnack inequality. Sections \ref{sec4} and \ref{sec5} are devoted to the proofs of Theorems \ref{thm:2d} and \ref{thm:2d:2}, respectively, for which we construct suitable sub- and super-solutions. In Section \ref{sec6}, we employ a Bernstein type argument to prove Theorem \ref{thm:bern}. Here we use the fact that for  any $q\geq p$, $|Du|^q$ is a subsolution to the normalized $p$-Laplace equation, as originally observed by Uhlenbeck \cite{MR474389}. Finally, we provide an alternative proof of the gradient estimates of order $\varepsilon^{-1/2}$ in the Appendix by also using the Bernstein type argument.

\section{Mean oscillation estimates}
In this section, we give the proof of Theorem \ref{thm-1/2} using mean oscillation estimates. We fix a point $x_0\in\Omega_{1/2}$ and prove \eqref{gradient-1/2} at $x=x_0$. Note that we can always assume $\varepsilon+|x_0'|^2\leq c$, where $c=c(n,p,c_1,c_2)>0$ could be any sufficiently small constant depending only on $n$, $p$, $c_1$, and $c_2$. Otherwise,  by classical estimates (see \cites{Lie,Vaz}), \eqref{gradient-1/2} directly follows.  Next, we derive some mean oscillation estimates of $Du$ on a ball $B_r(x_0)$ for different radii $r$.

\subsection{Mean oscillation estimates for small \texorpdfstring{$r$}{r}}

We recall a classical interior mean oscillation estimate when $B_{r}(x_0)\subset{\Omega}_1$. Estimates of this type, with different exponents involved, were developed in \cites{lieberman1991natural, dibenedetto1993higher, duzaar2010gradient}.
\begin{lemma}\label{lem:mean1}
Let $u\in W^{1,p}({\Omega}_1)$ be a solution to \eqref{main_problem_narrow}. There exist constants $C>1$ and $\alpha\in(0,1)$ depending only on $n$ and $p$, such that $u\in C^{1,\alpha}({\Omega}_1)$ and for every $B_{r}(x_0)\subset{\Omega}_1$ and $\rho\in (0,r]$, we have
\begin{align*}
    \left(\fint_{B_\rho(x_0)}|D u-(D u)_{B_\rho(x_0)}|^p \right)^{\frac{1}{p}}\leq C\left(\frac{\rho}{r}\right)^{\alpha } \left(\fint_{B_r(x_0)}|D u-(D u)_{B_r(x_0)}|^p \right)^\frac{1}{p}.
\end{align*}
\end{lemma}
We denote $$\phi(x_0,r)=\left(\fint_{B_r(x_0)}|D u-(D u)_{B_r(x_0)}|^p \right)^{\frac{1}{p}}.$$
Then Lemma \ref{lem:mean1} also implies
\begin{corollary}\label{cor:mean1}
Under the assumptions of Lemma \ref{lem:mean1}, there exists a constant  $\mu_1\in(0,1)$ depending only on $n$ and $p$,  such that for any $\mu\in(0,\mu_1]$, $B_r(x_0)\subset{\Omega}_1$, and $K\in\mathbb{N}$, it holds that
\begin{equation}\label{eq:holder1}
    \sum_{k=0}^{K+1} \phi(x_0, \mu^k r)\leq 2\,\phi(x_0, r).
\end{equation}
\end{corollary}
\begin{proof}
We take $\mu_1\in(0,1)$ such that $C\mu_1^\alpha= \frac{1}{2}$, where $C$ and $\alpha$ are the same constants as in Lemma \ref{lem:mean1}. Replacing $r$ with $\mu^k r$ and setting $\rho=\mu^{k+1}r$, we get
$$
\phi(x_0,\mu^{k+1} r)\leq \frac{1}{2} \phi(x_0,\mu^k r).
$$
Summing the above inequality over $k={0,1,\dots,K}$, we obtain \eqref{eq:holder1}.
\end{proof}

\subsection{Mean oscillation estimates for intermediate \texorpdfstring{$r$}{r}}
Next, we consider the case when $B_{r}(x_0)$ intersects with only one of $\Gamma_+$ and $\Gamma_-$. In this case, we choose $\hat{x}_0\in \Gamma_+\cup\Gamma_-$ such that $\text{dist}(x_0,\Gamma_+\cup\Gamma_-)=|\hat{x}_0-x_0|$ and we derive mean oscillation estimates around $\hat{x}_0$. Note that we can assume
$\varepsilon+c_2|x_0'|^2\leq 1/4$ and thus by \eqref{def:c_2} and the triangle inequality, $|\hat{x}_0'|\leq 3/4$.
Without loss of generality, we assume $\hat{x}_0\in \Gamma_-$. Then by \eqref{fg_1} and \eqref{def:c_2}, there exists a constant $c=c(n,c_1,c_2)\in(0,1/4)$, such that $B(\hat{x}_0,r)\cap \Gamma_+=\emptyset$ for any $r\in(0,c(\varepsilon+|\hat{x}_0'|^2))$.

We first choose a coordinate $y=(y',y_n)$ such that $y(\hat{x}_0)=0$, the direction of axis $y_n$ is the upper normal vector at $\hat{x}_0\in\Gamma_-$, and $\Omega_{R_0}(\hat{x}_0)=\{y\in B_{R_0}:y_n>\chi(y')\}$, where $R_0=c_3(\varepsilon+|\hat{x}_0'|^2)\in(0,1/4)$ for some constant $c_3=c_3(n,c_1,c_2)\in(0,1/8)$ and
$\chi:\{y'\in\mathbb{R}^{n-1}:|y'|<R_0\}\rightarrow\mathbb{R}$ is a $C^{1,1}$ function in the coordinate system depending on $\hat{x}_0$ such that
\begin{equation}\label{def:chi}
\chi(0')=0,\quad D_{y'}\chi(0^{\prime})=0,\quad \|\chi\|_{C^{1,1}}\leq C \|h_2\|_{C^{1,1}}.
\end{equation}
Then we let
$$
z=\Lambda(y) =(y',\,y_n-\chi(y')).
$$
Since $\Gamma_-$ is $C^{1,1}$, by \eqref{def:chi} there exist constants $C=C(n,c_1,c_2)$, $c_4=c_4(n,c_1,c_2)\in(0,c_3)\subset(0,1/8)$, and $R_1=c_4(\varepsilon+|\hat{x}_0'|^2)$ such that
\begin{equation}\label{chi}
   |D_{y'}\chi(y')|\leq C\,|y'|\leq 1/2 \quad \text{if} \quad |y'|\leq 2R_1,
\end{equation}
\begin{equation}\label{bdry}
    \Omega_{r/2}(\hat{x}_0)\subset \Lambda^{-1}(B^+_r)\subset \Omega_{2r}(\hat{x}_0) \quad \forall\, r\in(0,2R_1],
\end{equation}
and thus
\begin{equation}\label{co-diff}
   |D\Lambda(y)-I_n|\leq C\,|y'|\leq 1/2 \quad \text{if} \quad |y'|\leq 2R_1,
\end{equation}
Therefore, there exist positive constants $c(n)$ and $c'(n)$ depending only on $n$, such that for any $\hat{x}_0\in (\Gamma_+\cup\Gamma_-)\cap \{x\in\mathbb{R}^n: |x'|\leq 3/4\}$ and $0<r\leq c_4(\varepsilon+|\hat{x}_0'|^2)$,
\begin{equation}\label{volume1}
    c(n) r^n\leq  |\Omega_r(\hat{x}_0)| \leq c'(n)r^n.
\end{equation}
Note that
\begin{align}\label{det=1}
    \text{det}(D\Lambda)\equiv 1.
\end{align}
Then $u_1(z):=u(\Lambda^{-1}(z))$ satisfies the following equation with conormal boundary condition
\begin{equation}\label{eq:u1}
    \left\{
\begin{aligned}
     -\dv_z\left(|A^T D_z u_1|^{p-2}AA^T  D_z u_1\right)=\;&0 \quad  \text{in} \,\, B^+_{R_1},\\
   \left(|A^T D_z u_1|^{p-2}AA^T  D_z u_1\right)_n=\;&0 \quad \text{on}\,\, B_{R_1}\cap\partial \mathbb{R}^n_+,
\end{aligned}
\right.
\end{equation}
where we denote
$$
A:={A}(z):=(a_{ij}(z)):=D{\Lambda}(\Lambda^{-1}(z)) .
$$
Next we extend $u_1$ and the coefficient matrix $A$ to the whole ball $B_{R_1}$.
We take the even extension of $u_1$, $a_{nn}$, and $a_{ij},i,j = 1,2,\ldots, n-1,$ with respect to $z_n = 0$, and take the odd extension of $a_{in}$ and $a_{ni},i = 1,2,\ldots, n-1,$ with respect to $z_n = 0$. We still denote these functions by $u_1$ and ${A}$ after the extension. Because of the conormal boundary condition, $u_1$ satisfies
\begin{equation}\label{eq:u1:ext}
 -\dv_z\big(\mathbf{A}(z,D_z u_1)\big)=\;0 \quad  \text{in} \,\, B_{R_1},
\end{equation}
where the nonlinear operator $\mathbf{A}$ is defined as
$$
\mathbf{A}(z,\xi)=  |A^T\xi|^{p-2} AA^T\xi \quad \quad \text{for} \quad z\in B_{R_1},\,\,\xi\in \mathbb{R}^n.
$$

\begin{lemma}
There exists a constant $C=C(n,p,c_1,c_2)>0$, such that for any $z\in B_{R_1}$ and $\xi\in\mathbb{R}^n$,
\begin{align}\label{co-diff2}
    |\mathbf{A}(z,\xi)-|\xi|^{p-2}\xi|\leq C\,|z'|\,|\xi|^{p-1}.
\end{align}
\end{lemma}
\begin{proof}
We recall a well-known inequality (see \cite{MR1179331}*{Lemma 2.1}): for any $p>1$ and $\xi_1,\, \xi_2\in \mathbb{R}^n$, it holds that
\begin{equation}\label{V}
    c^{-1} \big(|\xi_1|^2+|\xi_2|^2)^\frac{p-2}{2}\leq \frac{\big||\xi_2|^{p-2}\xi_2-|\xi_1|^{p-2}\xi_1\big|}{|\xi_2-\xi_1|}\leq c \,\big(|\xi_1|^2+|\xi_2|^2)^\frac{p-2}{2},
\end{equation}
where $c=c(n,p)>1$ is a positive constant.
Using \eqref{V}, \eqref{co-diff}, and the triangle inequality, we obtain
\begin{align*}
    &|\mathbf{A}(z,\xi)-|\xi|^{p-2}\xi|\leq \big||A^T\xi|^{p-2} (A-I_n)A^T\xi\big|+\big|  |A^T\xi|^{p-2} A^T\xi-|\xi|^{p-2}\xi \big|\\
    &\leq |A-I_n|\,|A^T\xi|^{p-1}+c\,\big(|A^T\xi|^{2}+|\xi|^2\big)^\frac{p-2}{2} | A^T\xi-\xi|\leq C\, |z'|\,|\xi|^{p-1}.
\end{align*}
Thus the proof of \eqref{co-diff2} is completed.
\end{proof}
Assume that $r\in(0,R_1]$. We let $v_1\in u_1+W_0^{1,p}(B_r)$ be the unique solution to
\begin{equation}\label{eq:v1}
    \left\{
\begin{aligned}
     -\dv_z(|D_z v_1|^{p-2} D_z v_1) =\;&  0 \quad \;\;\text{in} \,\, B^+_{r},\\
     v_1 =\; &  u_1 \quad \text{on}\,\, \partial B^+_{r}. \\
\end{aligned}
\right.
\end{equation}
By testing  \eqref{eq:v1} and \eqref{eq:u1:ext} with $v_1-u_1$ and using \eqref{co-diff2}, we have the comparison estimate
\begin{align}\label{comp1}
    \fint_{B_r}|D_z u_1-D_z v_1|^p\leq Cr^{\min\{2,p}\}\fint_{B_{r}}|D_z u_1|^p,
\end{align}
where $C>0$ is a constant depending only on $n$, $p$, $c_1$, and $c_2$.
For detailed proof of \eqref{comp1}, see \cite{duzaar2010gradient}*{Eq. (4.35)} when $p\in(1,2)$ and \cite{MR2823872}*{Lemma 3.4} when $p\geq 2$.

Applying Lemma \ref{lem:mean1} and the comparison estimate \eqref{comp1}, we have
\begin{lemma}
Suppose that $u_1\in W^{1,p}(B^+_{R_1})$ is a solution to  \eqref{eq:u1}. Then for any $\mu\in(0,1)$ and $r\in (0,R_1]$, we have
\begin{equation}\label{ineq:pr}
\begin{aligned}
&\left(\fint_{B^+_{ \mu r}}|D_{z'} u_1-(D_{z'} u_1)_{B^+_{ \mu r}}|^{p}+|D_{z_n}u_1|^{p}\right)^{1/p}\\
&\leq C\mu^{\alpha}\left(\fint_{B^+_{ r}}|D_{z'} u_1-(D_{z'} u_1)_{B^+_{ r}}|^{p}+|D_{z_n}u_1|^{p}\right)^{1/p}+C_\mu r^{\theta_p}\left(\fint_{B^+_{r}}|D_z u_1|^{p}\right)^{{1}/{p}},
\end{aligned}
\end{equation}
where $\theta_p=\min\{1,2/p\}$, $\alpha$ is the same constant as in Lemma \ref{lem:mean1}, $C_\mu$ is a constant depending on $\mu$, $n$, $p$, $c_1$, $c_2$ and $C$ is a constant depending on $n$, $p$, $c_1$, $c_2$.
\end{lemma}
\begin{proof}
 By Lemma \ref{lem:mean1}, \eqref{comp1}, and the triangle inequality, we have
\begin{equation}\label{ineq:phi}
\begin{aligned}
        &\left(\fint_{B_{\mu r}}|D_z u_1-(D_z u_1)_{B_{\mu r}}|^{p}\right)^{1/p}
        \\
        &\leq C\left(\fint_{B_{\mu r}}|D_z v_1-(D_z v_1)_{B_{\mu r}}|^{p}\right)^{1/p}+C\left(\fint_{B_{\mu r}}|D_z u_1-D_z v_1|^{p}\right)^{1/p}\\
        &\leq C\mu^{\alpha}\left(\fint_{B_r}|D_z v_1-(D_z v_1)_{B_{r}}|^{p}\right)^{1/p}+ C\mu^{-\frac n{p}}\left(\fint_{B_r}|D_z u_1-D_z v_1|^{p}\right)^{1/p}\\
        &\leq C\mu^{\alpha}\left(\fint_{B_r}|D_z u_1-(D_z u_1)_{B_{r}}|^{p}\right)^{1/p}+ C\mu^{-\frac n{p}}\left(\fint_{B_r}|D_z u_1-D_z v_1|^{p}\right)^{1/p}\\
        &\leq C\mu^{\alpha}\left(\fint_{B_r}|D_z u_1-(D_z u_1)_{B_{r}}|^{p}\right)^{1/p}+C_\mu r^{\theta_p} \left(\fint_{B_{r}}|D_z u_1|^{p}\right)^{1/p}.
\end{aligned}
\end{equation}
Since $u_1$ is even in $z_n$, \eqref{ineq:phi} directly implies \eqref{ineq:pr}. The proof is completed.
\end{proof}
We now define
\begin{align}\label{def:psi}
\psi(\hat{x}_0,r)=\left(\fint_{\Omega_{r}(\hat{x}_0)}|D_{y'} u-(D_{y'} u)_{\Omega_r(\hat{x}_0)}|^{p}+|D_{y_n}u|^{p}\right)^{1/p}.
\end{align}
Let $\mu\in(0,1)$ and  $r\in (0, R_1/2]$ be constants. By using change of variables, \eqref{chi}, \eqref{bdry}, \eqref{det=1}, and the triangle inequality,  we have
\begin{equation}\label{eq:geq}
\begin{aligned}
    &\left(\fint_{B^+_{ \mu r}}|D_{z'} u_1-(D_{z'} u_1)_{B^+_{ \mu r}}|^{p}+|D_{z_n}u_1|^{p}\right)^{1/p}\\
    &= \left(\fint_{\Lambda^{-1}(B^+_{ \mu r})}|D_{y'} u+D_{y_n} u \,D_{y'}\chi-(D_{y'} u+D_{y_n} u \,D_{y'}\chi)_{\Lambda^{-1}(B^+_{ \mu r})}|^{p}+|D_{y_n} u|^{p}\right)^{1/p}\\
     &\geq C\left(\fint_{\Omega_{ \mu r/2}(\hat{x}_0)}|D_{y'} u-(D_{y'} u)_{\Omega_{ \mu r/2}(x_0)}|^{p}+|D_{y_n}u|^{p}\right)^{1/p}\\
     &\quad -C' \left(\fint_{\Omega_{\mu r/2}(\hat{x}_0)}|D_{y_n}u\; D_{y'}\chi|^{p}\right)^{1/p} \\
    &\geq C\psi(\hat{x}_0,\mu r/2)-C'\mu r\left(\fint_{\Omega_{\mu r/2}(\hat{x}_0)} |D u|^{p}\right)^{1/p},
\end{aligned}
\end{equation}
where $C$ and $C'$ are positive constants depending on $n$, $p$, $c_1$, and $c_2$.
Similarly,
\begin{equation}\label{eq:leq}
\begin{aligned}
    &\left(\fint_{B^+_{ r}}|D_{z'} u_1-(D_{z'} u_1)_{B^+_{ r}}|^{p}+|D_{z_n}u_1|^{p}\right)^{1/p}\\
    &\leq C''\psi(\hat{x}_0,2r)+C''r\left(\fint_{\Omega_{2r}(\hat{x}_0)} |D u|^{p}\right)^{1/p},
\end{aligned}
\end{equation}
where $C''$ is a positive constant depending only on $n$, $p$, $c_1$, and $c_2$. Therefore, by using \eqref{eq:geq}, \eqref{eq:leq}, and \eqref{volume1}, \eqref{ineq:pr} implies that
\begin{align*}
        &\psi(\hat{x}_0,\mu r/2)\leq C\mu^{\alpha}\psi(\hat{x}_0,2r)+
C_\mu r^{\theta_p}\left(\fint_{\Omega_{2r}(\hat{x}_0)}|D u|^{p}\right)^{1/p}.
\end{align*}
By replacing $\mu/4$ and $2r$ with $\mu$ and $r$ respectively, we obtain
\begin{align*}
       &\psi(\hat{x}_0,\mu r)\leq C\mu^{\alpha}\psi(\hat{x}_0, r)+C_\mu r^{\theta_p}\left(\fint_{\Omega_{r}(\hat{x}_0)}|D u|^{p}\right)^{1/p}
\end{align*}
for $\mu\in(0,1/4)$ and $r\in(0,R_1]$, where we recall that $R_1=c_4(\varepsilon+|\hat{x}_0'|^2)$.
Note that the same argument above also holds when $\hat{x}_0\in \Gamma_{+}$. Therefore, using the same argument as in Corollary \ref{cor:mean1}, we have
\begin{lemma}\label{lem:mean2}
Suppose that $u$ is a solution to \eqref{equzero} and $\hat{x}_0\in (\Gamma_+\cup\Gamma_-)\cap \{x\in\mathbb{R}^n: |x'|\leq 3/4\}$. Then there exist constants $c_4\in(0,1)$ and  $C>0$, both depending only on $n$, $p$, $c_1$, and $c_2$, and $C_\mu>0$ depending on $n$, $p$, $c_1$, $c_2$, and $\mu$, such that for any $\mu\in(0,1/4)$ and $r\in(0,\,c_4(\varepsilon +|\hat{x}_0'|^2)]$,  it holds that
\begin{equation*}
\begin{aligned}
       &\psi(\hat{x}_0,\mu r)\leq C\mu^{\alpha}\psi(\hat{x}_0, r)+C_\mu r^{\theta_p}\left(\fint_{\Omega_{r}(\hat{x}_0)}|D u|^{p}\right)^{1/p},
\end{aligned}
\end{equation*}
where $\theta_p=\min\{1,2/p\}$, $\alpha$ is the same constant as in Lemma \ref{lem:mean1}, and $\psi$ is defined in \eqref{def:psi}.
Moreover, there exist constants $\mu_2=\mu_2(n,p,c_1,c_2)\in(0,1/4)$ and $C_\mu'=C_\mu'(n,p,c_1,c_2,\mu)>0$, such that for any $\mu\in(0,\mu_2]$ and $K\in\mathbb{N}$, it holds that
$$
\sum_{k=0}^{K+1}\psi(\hat{x}_0,\mu^k r)\leq 2\,\psi(\hat{x}_0, r)+C_\mu' \sum_{k=0}^{K} (\mu^k r)^{\theta_p}\left(\fint_{\Omega_{\mu^k r}(\hat{x}_0)}|D u|^{p}\right)^{1/p}.
$$

\subsection{Mean oscillation estimates for large \texorpdfstring{$r$}{r}}
\end{lemma}
\subsection{Mean oscillation estimates for large \texorpdfstring{$r$}{r}}\label{sec1.4}
Finally, we consider the case when $B_{r}(x_0)$ could potentially intersects with both $\Gamma_+$ and $\Gamma_-$. In this case, we assume $x_0\in\Omega_{1/2}$ and  $\frac{c_4}{12} (\varepsilon +|x_0'|^2)\leq r\leq c_5(\varepsilon +|x_0'|^2)^{\frac{1}{2}}$, where $c_4$ is the same constant as in Lemma \ref{lem:mean2} and $c_5$ is a constant which will be determined later.
 We define the map $\mathcal{Z}=\Tilde{\Lambda}(x)$ by
\begin{equation*}
\left\{
\begin{aligned}
\mathcal{Z}' &= x'-x_0',\\
\mathcal{Z}_n &= (h_1(x_0')-h_2(x_0')+\varepsilon)\Big(\frac{x_n-h_2(x')+{\varepsilon}/{2}}{h_1(x')-h_2(x')+\varepsilon}-\frac{1}{2}\Big).
\end{aligned}
\right.
\end{equation*}
Thus $\Tilde{\Lambda}$ is invertible in ${\Omega}_{x_0,1/2}$,
$$
{Q}_{1/2}:=\Tilde{\Lambda}({\Omega}_{x_0,1/2})=\big\{(\mathcal{Z}',\mathcal{Z}_n)\in \mathbb{R}^n: \,|\mathcal{Z}'|<\frac{1}{2}, \,|\mathcal{Z}_n|<\frac{1}{2}(h_1(x_0')-h_2(x_0')+\varepsilon)\big\},
$$
and
\begin{align*}
     \Tilde{\Gamma}_{\pm}
     &:=\Tilde{\Lambda}(\Gamma_{\pm}\cap\{x\in\mathbb{R}^n:|x'-x_0'|<\frac{1}{2}\})\\
     &=\big\{(\mathcal{Z}',\mathcal{Z}_n)\in \mathbb{R}^n: \,|\mathcal{Z}'|<\frac{1}{2}, \,\mathcal{Z}_n=\pm \frac{1}{2}(h_1(x_0')-h_2(x_0')+\varepsilon)\big\}.
\end{align*}
Then $u_2(\mathcal{Z}):=u(\Tilde{\Lambda}^{-1}(\mathcal{Z}))$ satisfies the following equation with homogeneous conormal boundary condition
\begin{equation*}
    \left\{
\begin{aligned}
     -\dv_\mathcal{Z}\left(|B^T D_\mathcal{Z} u_2|^{p-2}(\text{det}( B))^{-1}BB^T  D_\mathcal{Z} u_2\right)=\;&0 \quad  \text{in} \,\, {Q}_{1/2}\\
   \left(|B^T D_\mathcal{Z} u_2|^{p-2}(\text{det}( B))^{-1}BB^T  \,D_\mathcal{Z} u_2\right)_n=\;&0 \quad \text{on}\,\, \Tilde{\Gamma}_{\pm},
\end{aligned}
\right.
\end{equation*}
where we denote
$$
B:=B(\mathcal{Z}):=(b_{ij}(\mathcal{Z})):=D{\Tilde{\Lambda}}(\Tilde{\Lambda}^{-1}(\mathcal{Z})).
$$

For $\mathcal{Z}\in Q_{1/2}$, let $x=\Tilde{\Lambda}^{-1}(\mathcal{Z})$.
 Then
$$b_{ii}(\mathcal{Z})=1 \quad \text{for}  \;i\in\{1,2,\ldots,n-1\},$$
$$ b_{ij}(\mathcal{Z})=0 \quad \text{for}\; i\neq j, \; i\in\{1,2,\ldots,n-1\},\;j\in\{1,2,\ldots,n\},
$$
\begin{align*}
    b_{nj}(\mathcal{Z})&=\frac{h_1(x_0')-h_2(x_0')+\varepsilon}
    {(h_1(x')-h_2(x')+\varepsilon)^2}\\
    &\quad \cdot \big[D_{x_j}h_2(x')\big(x_n-h_1(x')-\frac{\varepsilon}{2}\big)-D_{x_j} h_1(x')\big(x_n-h_2(x')+\frac{\varepsilon}{2}\big)\big]
\end{align*}
for $j\in\{1,2,\ldots,n-1\}$, and
$$
b_{nn}(\mathcal{Z})=\frac{h_1(x_0')-h_2(x_0')+\varepsilon}{h_1(x')-h_2(x')+\varepsilon}.
$$
Therefore,
\begin{align*}
    \text{det}(B(\mathcal{Z}))=b_{nn}(\mathcal{Z})=\frac{h_1(x_0')-h_2(x_0')+\varepsilon}{h_1(x_0'+\mathcal{Z}')-h_2(x_0'+\mathcal{Z}')+\varepsilon}
\end{align*}
is a function independent of $\mathcal{Z}_n$.
Assume
\begin{equation}
                        \label{eq11.05}
\frac{c_4}{12} (\varepsilon +|x_0'|^2)\leq r\leq \frac{1}{4}(\varepsilon +|x_0'|^2)^{\frac{1}{2}}\leq \frac{1}{2}
\end{equation}
and let $\mathcal{Z}_0= \Tilde{\Lambda}(x_0)$.
Then for any $\mathcal{Z}\in Q_{1/2}$ with $|\mathcal{Z}'|\leq r$ and $x=\Tilde{\Lambda}^{-1}(\mathcal{Z})$, by the triangle inequality, we have
$$
|x'|\leq r+|x_0'|\leq \big(1+\sqrt{12/c_4}\big)r^{\frac{1}{2}} \quad \text{and} \quad |x'|^2\geq \frac{1}{2}|x_0'|^2-r^2\geq \frac{1}{4} (|x_0'|^2-\varepsilon).
$$
Thus, using \eqref{fg_0}, \eqref{fg_1}, and \eqref{def:c_2}, we infer that for $j=1,2,\ldots,n-1$ and some constant $C>0$ depending only on $n$, $p$, $c_1$, and $c_2$,
\begin{equation*}
    \begin{aligned}
    &|b_{nj}(\mathcal{Z})|\leq 2c_2\frac{|x'|(h_1(x_0')-h_2(x_0')+\varepsilon)}{h_1(x')-h_2(x')+\varepsilon}\leq 2c_2 \frac{|x'|(2c_2|x_0'|^2 +\varepsilon)}{c_1|x'|^2+\varepsilon}
    \\&\leq C|x'|\leq Cr^{\frac{1}{2}}
    \leq \frac{Cr}{(\varepsilon+|x_0'|^2)^\frac{1}{2}},
\end{aligned}
\end{equation*}
\begin{equation*}
    \begin{aligned}
    &|b_{nn}(\mathcal{Z})-1|=\Big|\frac{\int_0^1 \frac{d}{dt}\big(h_1(tx'+(1-t)x_0')-h_2(tx'+(1-t)x_0')\big)dt}{h_1(x')-h_2(x')+\varepsilon}\Big|\\
    &\leq 2c_2\frac{(|x'|+|x_0'|)|x'-x_0'|}{c_1|x'|^2+\varepsilon}\leq \frac{Cr}{(\varepsilon+|x_0'|^2)^\frac{1}{2}},
\end{aligned}
\end{equation*}
and similarly,
\begin{equation}\label{det-1}
\begin{aligned}
 &\big|\big(\text{det}(B(\mathcal{Z}))\big)^{-1}-1\big|=\big|\big(b_{nn}(\mathcal{Z})\big)^{-1}-1\big|\leq 
 \frac{Cr}{(\varepsilon+|x_0'|^2)^\frac{1}{2}}.
\end{aligned}
\end{equation}

Therefore, when \eqref{eq11.05} holds and $\mathcal{Z}\in Q_{1/2}$ with $|\mathcal{Z}'|\leq r$, we have for some constant $C=C(n,p,c_1,c_2)>0$,
\begin{equation}\label{co-diff3}
   |B(\mathcal{Z})-I_n|\leq \frac{Cr}{(\varepsilon+|x_0'|^2)^\frac{1}{2}}.
\end{equation}
In particular, there exists $c_5=c_5(n,p,c_1,c_2)\in(0,1/4)$, such that for any $\frac{c_4}{12} (\varepsilon +|x_0'|^2)\leq r\leq c_5(\varepsilon +|x_0'|^2)^{\frac{1}{2}}$
and $\mathcal{Z}\in Q_{1/2}$ with $|\mathcal{Z}'|\leq r$, it also holds that
\begin{align}\label{co-diff4}
    |B(\mathcal{Z})-I_n|\leq 1/2\quad \text{and} \quad \big|\big(b_{nn}(\mathcal{Z})\big)^{-1}-1\big|\leq 1/2.
\end{align}
Note that we can always assume $\varepsilon +|x_0'|^2$ to be sufficiently small so that $c_4 (\varepsilon +|x_0'|^2)\leq c_5(\varepsilon +|x_0'|^2)^{\frac{1}{2}}$. Next we extend $u_2$ and ${B}$ to the whole cylinder $\mathcal{C}_{1/2}:=\{(\mathcal{Z}',\mathcal{Z}_n)\in \mathbb{R}^n: \,|\mathcal{Z}'|<1/2\}$.
We take the even extension of $u_2$, $b_{nn}$, and $b_{ij}, i,j = 1,2,\ldots, n-1,$ with respect to $\mathcal{Z}_n = \frac{1}{2}(h_1(x_0')-h_2(x_0')+\varepsilon)$, and take the odd extension of $b_{in}$ and $b_{ni}, i = 1,2,\ldots, n-1,$ with respect to $\mathcal{Z}_n = \frac{1}{2}(h_1(x_0')-h_2(x_0')+\varepsilon)$. Then we take the periodic extension in $\mathcal{Z}_n$ axis, so that the period is equal to $2(h_1(x_0')-h_2(x_0')+\varepsilon)$. We still denote these functions by $u_2$ and ${B}$ after the extension. Then because of the conormal boundary condition, $u_2$ satisfies
\begin{equation}\label{eq:u3}
 -\dv_\mathcal{Z}\big(\mathbf{B}(\mathcal{Z},D_\mathcal{Z} u_2)\big)=\;0 \quad  \text{in} \;\mathcal{C}_{1/2},
\end{equation}
where the nonlinear operator $\mathbf{B}$ is defined as
$$
\mathbf{B}(\mathcal{Z},\xi)= d(\mathcal{Z}')| B^T \xi|^{p-2} BB^T\xi  \quad \quad \text{for} \quad \mathcal{Z}\in\mathcal{C}_{1/2},\,\xi\in \mathbb{R}^n,
$$
and
$$d(\mathcal{Z}'):=\big(b_{nn}(\mathcal{Z})\big)^{-1}=\frac{h_1(\mathcal{Z}'+x_0')-h_2(\mathcal{Z}'+x_0')+\varepsilon}{h_1(x_0')-h_2(x_0')+\varepsilon}.
$$

Similar to \eqref{co-diff2}, using \eqref{det-1}, \eqref{co-diff3}, \eqref{co-diff4}, and \eqref{V}, we obtain that  for any $r\in \big[\frac{c_4}{12} (\varepsilon +|x_0'|^2),c_5(\varepsilon +|x_0'|^2)^{\frac{1}{2}}\big]$, $\mathcal{Z}\in B_{r}(\mathcal{Z}_0)$, and $\xi\in\mathbb{R}^n$,
\begin{align}\label{co-diff5}
    |\mathbf{B}(\mathcal{Z},\xi)-|\xi|^{p-2}\xi|\leq \frac{Cr}{(\varepsilon+|x_0'|^2)^\frac{1}{2}}|\xi|^{p-1},
\end{align}
where $C>0$ is a constant depending only on $n$, $p$, $c_1$, and $c_2$.
Now we let $v_2\in u_2+W_0^{1,p}(B_r(\mathcal{Z}_0))$ be the unique solution to

\begin{equation*}
    \left\{
\begin{aligned}
     -\dv_\mathcal{Z}\big(|D_\mathcal{Z} v_2|^{p-2} D_\mathcal{Z} v_2\big) =\;&  0 \quad \;\;\text{in} \,\, B_{r}(\mathcal{Z}_0),  \\
     v_2 =\; &  u_2 \quad \text{on}\,\, \partial B_{r}(\mathcal{Z}_0).
\end{aligned}
\right.
\end{equation*}
Using \eqref{co-diff5}, similar to \eqref{comp1}, we have the following comparison estimate
\begin{align}\label{comp2}
    \fint_{B_r(\mathcal{Z}_0)}|D_\mathcal{Z} u_2-D_\mathcal{Z} v_2|^p\leq C\Big(\frac{r}{(\varepsilon+|x_0'|^2)^\frac{1}{2}}\Big)^{\min\{2,p\}}\fint_{B_{r}(\mathcal{Z}_0)}|D_\mathcal{Z} u_2|^p,
\end{align}
where $C>0$ is a constant depending only on $n$, $p$, $c_1$, and $c_2$.

We  define
\begin{align}\label{def:phi2}
\Tilde{\phi}(x_0,r)=\left(\fint_{B_{r}(\mathcal{Z}_0)}|D_{\mathcal{Z}} u_2-(D_{\mathcal{Z}} u_2)_{B_r(\mathcal{Z}_0)}|^{p}\right)^{1/p}.
\end{align}
Then following the same proof as that of Lemma \ref{lem:mean2} with \eqref{comp2} in place of \eqref{comp1}, we have
\begin{lemma}\label{lem:mean3}
Suppose that $x_0\in\Omega_{1/2}$ and  $u_2$ is a solution to \eqref{eq:u3}. Then there exist constants $c_5\in(0,1/4)$ and  $C>0$, both depending only on $n$, $p$, $c_1$, and $c_2$, and $C_\mu>0$ depending on $n$, $p$, $c_1$, $c_2$, and $\mu$, such that for any $\mu\in(0,1)$ and $r\in\big[\frac{c_4}{12}(\varepsilon +|x_0'|^2),\,c_5(\varepsilon +|x_0'|^2)^{\frac{1}{2}}]$, it holds that
\begin{equation*}
\begin{aligned}
\Tilde{\phi}(x_0,\mu r)\leq C\mu^\alpha \Tilde{\phi}(x_0,r)+C_\mu\Big(\frac{ r}{(\varepsilon+|x_0'|^2)^\frac{1}{2}}\Big)^{\theta_p}\left(\fint_{B_{r}(\mathcal{Z}_0)}|D_\mathcal{Z} u_2|^{p}\right)^{1/p},
\end{aligned}
\end{equation*}
where $\theta_p=\min\{1,2/p\}$, $\alpha$ is the same constant as in Lemma \ref{lem:mean1}, $c_4$ is the same constant as in Lemma \ref{lem:mean2}, and $\Tilde{\phi}$ is defined in \eqref{def:phi2}.
Moreover, there exist constants $\mu_3=\mu_3(n,p,c_1,c_2)\in(0,1)$ and $C_\mu'=C_\mu'(n,p,c_1,c_2,\mu)>0$, such that for any $\mu\in(0,\mu_3]$ and $k_1,\, k_2\in\mathbb{N}$ satisfying $\frac{c_4}{12}(\varepsilon +|x_0'|^2)\leq \mu ^{k_2}r\leq \mu^{k_1}r\leq c_5(\varepsilon +|x_0'|^2)^{\frac{1}{2}}$,  it holds that
$$
\sum_{k=k_1}^{k_2+1}\Tilde{\phi}(x_0,\mu^k r)\leq 2\,\Tilde{\phi}(x_0, \mu^{k_1}r)+C_\mu'\sum_{k=k_1}^{k_2} \Big(\frac{ \mu^k r}{(\varepsilon+|x_0'|^2)^\frac{1}{2}}\Big)^{\theta_p}\left(\fint_{B_{\mu^k r}(\mathcal{Z}_0)}|D_\mathcal{Z} u_2|^{p}\right)^\frac{1}{p}.
$$
\end{lemma}

\subsection{Proof of Theorem \ref{thm-1/2}}
Now we are ready to prove Theorem \ref{thm-1/2}.

\begin{proof}[Proof of Theorem \ref{thm-1/2}]
We prove the theorem at any point $x_0\in \Omega_{1/2}$.

\emph{Step 1: Notation and choices of constants.}
 We choose $\mu=\frac{1}{6}\min\{\mu_1,\mu_2,\mu_3\}$, where $\mu_1$, $\mu_2$, and $\mu_3$ are the same constants as in Corollary \ref{cor:mean1},  Lemma \ref{lem:mean2}, and Lemma \ref{lem:mean3}.
We define $r_j=\frac{c_5}{2}\mu^j (\varepsilon+|x_0'|^2)^\frac{1}{2}$ and let $j_1$, $j_2$ be the integers such that $$r_{j_1}\geq \frac{c_4}{6} (\varepsilon+|x_0'|^2),\quad r_{j_1+1}<\frac{c_4}{6}(\varepsilon+|x_0'|^2),$$
and
$$r_{j_2}\geq \text{dist}(x_0,\Gamma_+\cup\Gamma_-), \quad r_{j_2+1}<\text{dist}(x_0,\Gamma_+\cup \Gamma_-),$$
where $c_4$ and $c_5$ are the same constants as in Lemma \ref{lem:mean2} and Lemma \ref{lem:mean3}. Note that we can assume $\varepsilon+|x_0'|^2$ to be sufficiently small so that $\frac{c_5}{2}(\varepsilon+|x_0'|^2)^\frac{1}{2}>\frac{c_4}{6} (\varepsilon+|x_0'|^2)$ and thus $j_1\geq 0$.

We denote
\begin{equation*}
    \phi_j=\left\{
\begin{aligned}
     & \left(\fint_{B_{ r_j}(\mathcal{Z}_0)}|D_{\mathcal{Z}} u_2-(D_{\mathcal{Z}} u_2)_{B_{ r_j}(\mathcal{Z}_0)}|^{p}\right)^{1/p} \quad \quad \,\text{if} \quad 0\leq j\leq j_1,\\
     & \left(\fint_{\Omega_{ r_j}(x_0)}|D u-(D u)_{\Omega_{ r_j}(x_0)}|^{p}\right)^{1/p} \quad  \,\;\text{if} \quad j\geq j_1+1,
\end{aligned}
\right.
\end{equation*}

\begin{equation*}
    T_j=\left\{
\begin{aligned}
     & \left(\fint_{B_{ r_j}(\mathcal{Z}_0)}|D_{\mathcal{Z}} u_2|^{p}\right)^{1/p} \quad \;\;\;\text{if} \quad 0\leq j\leq j_1,\\
     & \left(\fint_{\Omega_{ 3r_j}(x_0)}|D u|^{p}\right)^{1/p} \quad \,\text{if} \quad j\geq j_1+1,
\end{aligned}
\right.
\end{equation*}
and
\begin{equation*}
    \mathbf{m}_j=\left\{
\begin{aligned}
     & (D_{\mathcal{Z}} u_2)_{B_{ r_j}(\mathcal{Z}_0)}\quad \;\;\text{if} \quad 0\leq j\leq j_1,\\
     & (D u)_{\Omega_{ r_j}(x_0)} \quad  \,\text{if} \quad j\geq j_1+1,
\end{aligned}
\right.
\end{equation*}
where  $u_2$, $\mathcal{Z}_0$, and the coordinate $\mathcal{Z}$ are defined in Section \ref{sec1.4}.
In the following proof, we use $C$, $C'$ to denote positive constants depending only on $n$, $p$, $c_1$, and $c_2$, which may differ from line to line.

\emph{Step 2: Preliminary estimates and iterations.}
Next, we derive some preliminary estimates.
First, we show that there exists a constant $c=c(n)>0$, such that for any $j\geq j_1+1$,
\begin{align}\label{volume}
    |\Omega_{r_j}(x_0)|\geq c \,r_j^n.
\end{align}
If $r_j\leq 2 \,\text{dist}(x_0,\Gamma_+\cup \Gamma_-)$, then
$B_{\frac{1}{2}r_j}(x_0)\subset\Omega_{r_j}(x_0)$ and \eqref{volume} clearly holds. Otherwise, assume $r_j>2 \,\text{dist}(x_0,\Gamma_+\cup \Gamma_-)$. Then we choose $\hat{x}_0\in \Gamma_+\cup\Gamma_-$ such that $\text{dist}(x_0,\Gamma_+\cup\Gamma_-)=|\hat{x}_0-x_0|$, and thus $\Omega_{\frac{1}{2}r_j}(\hat{x}_0)\subset \Omega_{r_j}(x_0)$. Note that we can assume
$\varepsilon+c_2|x_0'|^2\leq 1/4$ and thus by \eqref{def:c_2} and the triangle inequality, $|\hat{x}_0'|\leq 3/4$. Since $j\geq j_1+1$, by the triangle inequality again, we know that $|x_0'|\le |\hat{x}_0'|+r_j/2$ and
$$r_j< \frac{c_4}{6} (\varepsilon+|x_0'|^2)\leq  \frac{c_4}{6} \big(\varepsilon+2|\hat{x}_0'|^2+\frac{1}{2}r_j^2\big).$$
Since $c_4\in (0,1)$ and $r_j\in(0,1)$, we get
\begin{align*}
r_j< \frac{c_4}{2}(\varepsilon+|\hat{x}_0'|^2).
\end{align*}
By using \eqref{volume1}, we have
$$|\Omega_{r_j}(x_0)|\geq |\Omega_{\frac{1}{2}r_j}(\hat{x}_0)|\geq c\, r_j^n.$$
Thus, \eqref{volume} holds for every $j\geq j_1+1$.

By \eqref{volume} and H\"older's inequality, for any $j\in \mathbb{N}$, we have
\begin{align}\label{mbound}
     |\mathbf{m}_j|\leq C T_j.
\end{align}
Note that since $\mu\leq 1/6$, $\Omega_{3r_{j+1}}(x_0)\subset \Omega_{\frac{1}{2}r_j}(x_0)$ and thus by \eqref{volume}, \eqref{co-diff4} and the definition of $j_1$,
\begin{align*}
    T_{j_1+1}\leq C\left(\fint_{\Omega_{ \frac{1}{2}r_{j_1}}(x_0)}|D u|^{p}\right)^{1/p} \leq C\, T_{j_1}.
\end{align*}
Therefore, there exists $c_6=c_6(n,p,c_1,c_2)>0$, such that for any $j\in\mathbb{N}$,
\begin{align}\label{def:c_6}
    T_{j+1}\leq c_6\, T_{j}.
\end{align}
By \eqref{def:c_6} and the triangle inequality, for any $j\leq j_1$, we have
$$
T_{j+1}\leq c_6\,T_{j}\leq C\,|\mathbf{m}_j|+C\,\phi_j.
$$
For $j\geq j_1+1$, since $\mu\leq 1/6$, $\Omega_{3r_{j+1}}(x_0)\subset \Omega_{r_j}(x_0)$ and thus by \eqref{volume} and the triangle inequality, we have
$$
T_{j+1}\leq C\,\left(\fint_{\Omega_{ r_j}(x_0)}|D u|^{p}\right)^{1/p}\leq C\,|\mathbf{m}_j|+C\,\phi_j.
$$
Therefore, there exists $c_7=c_7(n,p,c_1,c_2)>0$, such that for any $j\in\mathbb{N}$,
\begin{align}\label{def:c_7}
    T_{j+1}\leq c_7\,|\mathbf{m}_j|+c_7\,\phi_j.
\end{align}
For any $0\leq k\leq j_1$, since
\begin{align*}
   |\mathbf{m}_{k}-\mathbf{m}_{k-1}|^{p}\leq 2^{p-1}|\mathbf{m}_{k}-D_\mathcal{Z} u_2(\mathcal{Z})|^{p}+2^{p-1}|D_\mathcal{Z} u_2(\mathcal{Z})-\mathbf{m}_{k-1}|^{p},
\end{align*}
by taking the average over $\mathcal{Z}\in B_{r_{k}}(\mathcal{Z}_0)$ and then taking the $p$-th root, we obtain
\begin{align*}
    |\mathbf{m}_{k}-\mathbf{m}_{k-1}|\leq C\phi_k+C\phi_{k-1}.
\end{align*}
Then by iterating, we get
\begin{align}\label{mdiff1}
     |\mathbf{m}_{j}-\mathbf{m}_{j_0}|\leq C \sum_{k=j_0}^{j}\phi_k,
\end{align}
for any integers $j_0$, $j$ satisfying $0\leq j_0\leq j\leq j_1$.
By \eqref{mdiff1}, \eqref{mbound}, and the triangle inequality, we have
\begin{align}\label{mbound2}
     |\mathbf{m}_{j}|\leq C\,T_{j_0} + C \sum_{k=j_0}^{j}\phi_k,
\end{align}
for $j_0\leq j\leq j_1$.
Similarly, for any integers $j$, $l$ satisfying $j_1+1\leq l\leq j$, we also have
\begin{align*}
    |\mathbf{m}_{j}-\mathbf{m}_{l}|\leq C \sum_{k=l}^{j}\phi_k,
\end{align*}
and
\begin{align}\label{mbound3}
    |\mathbf{m}_{j}|\leq C \,T_{l}+C \sum_{k=l}^{j}\phi_k.
\end{align}
For $j\in\{j_0, \ldots, j_1\}$, from Lemma \ref{lem:mean3} and \eqref{mbound2}, we know that
\begin{align}\label{big}
    |\mathbf{m}_j|+\sum_{k=j_0}^{j} \phi_k\leq C\,T_{j_0}+C\, \sum_{k=j_0}^{j} \Big(\frac{ r_k}{(\varepsilon+|x_0'|^2)^\frac{1}{2}}\Big)^{\theta_p} T_k\leq C\,T_{j_0}+C\, \sum_{k=j_0}^{j} \mu^{k\theta_p} T_k.
\end{align}
For $j\in\{j_1+1,\ldots,j_2\}$, we have $r_{j}\geq \text{dist}(x_0,\Gamma_+\cup\Gamma_-)$. Choose $\hat{x}_0\in (\Gamma_+\cup\Gamma_-)\cap \{x\in\mathbb{R}^n: |x'|\leq 3/4\}$ such that $\text{dist}(x_0,\Gamma_+\cup\Gamma_-)=|\hat{x}_0-x_0|$, and thus $\Omega_{r_j}(x_0)\subset \Omega_{2r_j}(\hat{x}_0)\subset \Omega_{3r_j}(x_0)$.
Then $$r_j< \frac{c_4}{6} (\varepsilon+|x_0'|^2)\leq  \frac{c_4}{6} (\varepsilon+2|\hat{x}_0'|^2+2r_j^2),$$
which also implies
\begin{align}\label{cond:mean2}
    2r_j< c_4(\varepsilon+|\hat{x}_0'|^2)
\end{align}
since $c_4\in (0,1)$ and $r_j\in(0,1)$.

 By \eqref{cond:mean2}, we can apply Lemma \ref{lem:mean2} at
 $\hat{x}_0\in\Gamma_+\cup\Gamma_-$ and use \eqref{volume} and \eqref{volume1} to obtain
\begin{equation}\label{mid_1}
\begin{aligned}
\sum_{k=j_1+1}^{j}\phi_k&\leq C \sum_{k=j_1+1}^{j}\psi(\hat{x}_0,2r_k)
\leq  C\,Y_{j_1+1}+
C\sum_{k=j_1+1}^{j}r_k^{\theta_p}Y_k
\\&\leq  C\,T_{j_1+1}
+C\sum_{k=j_1+1}^{j}\mu^{k\theta_p}T_k,
\end{aligned}
\end{equation}
where
$$
Y_j:=\left(\fint_{\Omega_{2r_j}(\hat{x}_0)}|D u|^{p}\right)^{1/p}.
$$
Moreover, from \eqref{mbound3}, \eqref{def:c_7}, and \eqref{mid_1} we also know that
\begin{equation}\label{mid}
\begin{aligned}
&|\mathbf{m}_j|+\sum_{k=j_1+1}^{j}\phi_j\leq  C\,T_{j_1+1}
+C\sum_{k=j_1+1}^{j}\mu^{k\theta_p}T_k\\
&\leq C\,|\mathbf{m}_{j_1}|+C\,\phi_{j_1}+C\sum_{k=j_1+1}^{j}\mu^{k\theta_p}T_k
\end{aligned}
\end{equation}
holds for any $j\in\{j_1+1,\ldots,j_2\}$.

For $j\geq j_2+1$, from Corollary \ref{cor:mean1}, \eqref{mbound3}, and \eqref{def:c_7}  we have
\begin{align}\label{small}
    |\mathbf{m}_j|+\sum_{k=j_2+1}^{j} \phi_k\leq C\,T_{j_2+1}\leq C \,|\mathbf{m}_{j_2}|+C\,\phi_{j_2}.
\end{align}
Combining \eqref{mid} and \eqref{small}, we know that
\begin{align}\label{mid_small}
    |\mathbf{m}_j|+\sum_{k=j_1+1}^{j}\phi_j&\leq C\,|\mathbf{m}_{j_1}|+C\,\phi_{j_1}+C\sum_{k=j_1+1}^{j}\mu^{k\theta_p}T_k
\end{align}
holds for any $j\geq j_1+1$.
Note that \eqref{mid_small} also holds if $j_2\leq j_1$ since in that case $r_j\leq  \text{dist}(x_0,\Gamma_+\cup\Gamma_-)$ for any $j\geq j_1+1$ and thus we can directly use Corollary \ref{cor:mean1} and \eqref{mbound3} to get \eqref{mid_small}.

Moreover, combining \eqref{mid_small} and \eqref{big}, we know that
\begin{equation}\label{all_radii}
|\mathbf{m}_j|+\sum_{k=j_0}^{j} \phi_k\leq C\,T_{j_0}+C\,  \sum_{k=j_0}^{j} \mu^{k\theta_p} T_k
\end{equation}
holds for any $0\leq j_0\leq j_1$ and $j\geq j_0$.

\emph{Step 3: A stopping time argument.}
We choose $j_0=j_0(n,p,c_1,c_2)\in\mathbb{N}$ sufficiently large such that
\begin{equation}\label{j_0}
    (c_7+1)\,C\sum_{k=j_0}^\infty \mu^{k\theta_p}\leq\frac{1}{10},
\end{equation}
where $c_7$ is the constant in \eqref{def:c_7} and $C$ is the constant in \eqref{all_radii}.
Note that we can assume $\varepsilon+|x_0'|^2$ to be sufficiently small so that $\frac{c_5}{2} \mu^{j_0}(\varepsilon+|x_0'|^2)^\frac{1}{2}>\frac{c_4}{6} (\varepsilon+|x_0'|^2)$ and thus $j_1\geq j_0$.
Now we show that
\begin{equation}\label{claim1}
 |D u(x_0)|\leq C\, T_{j_0}.
\end{equation}

We consider the following three possibilities.

{\em Case 1:} If $|D u(x_0)|\leq T_{j_0}$, then \eqref{claim1} directly follows.

{\em Case 2:} If $ T_j< |D u(x_0)|, \ \forall j_0\leq j\leq j_3$, and $|D u(x_0)|\leq T_{j_3+1}$, then by \eqref{def:c_7}, we have
\begin{equation}\label{ineq:grad}
\begin{aligned}
       |D u(x_0)|\leq T_{j_3+1}\leq c_7\,|\mathbf{m}_{j_3}|+c_7\,\phi_{j_3}.
\end{aligned}
\end{equation}
Now applying \eqref{all_radii} with $j=j_3$, from \eqref{ineq:grad} and \eqref{j_0}, we get
\begin{align*}
 |D u(x_0)|\leq C'\,T_{j_0}+C'\,  \sum_{k=j_0}^{j_3} \mu^{k\theta_p}  |D u(x_0)|
 \leq C'\,T_{j_0}+\frac{1}{10}|D u(x_0)|,
\end{align*}
where $C'=c_7\,C$, $C$ is the constant in \eqref{all_radii}. The last inequality directly implies \eqref{claim1} as desired.

{\em Case 3:} If $T_j<|D u(x_0)|$ for every $j\geq j_0$, then from \eqref{all_radii}, we infer that for any $j\geq j_0$,
\begin{align*}
    |\mathbf{m}_j|\leq C\,T_{j_0}+C\,  \sum_{k=j_0}^{j} \mu^{k\theta_p}|D u(x_0)|
    \leq C\,T_{j_0}+\frac{1}{10}|D u(x_0)|.
\end{align*}
Here we used \eqref{j_0} in the last inequality. Letting $j\to \infty$ and using the fact that $u\in C^1({\Omega}_1)$, we get
\begin{align*}
    |D u(x_0)| \leq C\,T_{j_0}+\frac{1}{10}|D u(x_0)|,
\end{align*}
which directly implies \eqref{claim1}.
The proof of the inequality \eqref{claim1} is completed.

\emph{Step 4: Caccioppoli inequality and conclusion.}
Let $\lambda\in\mathbb{R}$ and $\zeta$ be a nonnegative smooth function satisfying $\zeta=1$ in $B_{r_{j_0}}(\mathcal{Z}_0)$, $|D_\mathcal{Z} \zeta|\leq 2r_{j_0}^{-1}$, and $\zeta=0$ outside $B_{2r_{j_0}}(\mathcal{Z}_0)$.
Since $2r_{j_0}\leq c_5(\varepsilon+|x_0'|^2)^{\frac{1}{2}}$,
using $\zeta^p (u_2-\lambda)$ as a test function in \eqref{eq:u3}, by \eqref{co-diff4}, Young's inequality, we  obtain
\begin{align*}
&\frac{1}{2^{p+1}}\int_{B_{2r_{j_0}}(\mathcal{Z}_0)} \zeta^p|D_\mathcal{Z} u_2|^p \leq \int_{B_{2r_{j_0}}(\mathcal{Z}_0)} \langle \mathbf{B}(\mathcal{Z}, D_\mathcal{Z} u_2),\zeta^p D_\mathcal{Z} u_2\rangle
\\&=-p \int_{B_{2r_{j_0}}(\mathcal{Z}_0)} \langle \mathbf{B}(\mathcal{Z}, D_\mathcal{Z} u_2), \zeta^{p-1}  (u_2-\lambda)D_\mathcal{Z}\zeta \rangle\\ &
\leq p2^{p+2}r_{j_0}^{-1}\int_{B_{2r_{j_0}}(\mathcal{Z}_0)} \zeta^{p-1}|D_\mathcal{Z} u_2|^{p-1} |u_2-\lambda|
\\&\leq \frac{1}{2^{p+2}} \int_{B_{2r_{j_0}}} \zeta^p|D_\mathcal{Z} u_2|^p +c(p)\,r_{j_0}^{-p} \int_{B_{2r_{j_0}}(\mathcal{Z}_0)} |u_2-\lambda|.
\end{align*}
Therefore, we have the following Caccioppoli inequality
\begin{align}\label{caccio}
    \int_{B_{ r_{j_0}}(\mathcal{Z}_0)}|D_{\mathcal{Z}} u_2|^{p}\leq {c(p)}\,r_{j_0}^{-p} \int_{B_{ 2r_{j_0}}(\mathcal{Z}_0)}| u_2-\lambda|^{p},
\end{align}
where $\lambda$ is an arbitrary constant and $c(p)$ is a positive constant depending only on $p$.
Since $2r_{j_0}\leq c_5(\varepsilon+|x_0'|^2)^{\frac{1}{2}}\leq \frac{1}{4}(\varepsilon+|x_0'|^2)^{\frac{1}{2}}\leq 1/2$, by choosing $\lambda =(u_2)_{B_{ 2r_{j_0}}(\mathcal{Z}_0)}$ in \eqref{caccio}, and using \eqref{claim1}, we obtain the pointwise blow-up estimate
\begin{align*}
    |D u(x_0)|\leq C (\varepsilon+|x_0'|^2)^{-\frac{1}{2}}\underset{\Omega_{x_0,\eta}}{\osc}~u,
\end{align*}
where $\eta=\frac{1}{4}(\varepsilon+|x_0'|^2)^{\frac{1}{2}}$.
\end{proof}

\section{Improved gradient estimates}\label{sec3}
In this section, we utilize a similar approach of flattening the boundaries and extending the equation, as described in \cite{LY2}, to derive an improved gradient estimate for \eqref{main_problem_narrow} in dimensions $n \ge 3$. However, in contrast to \cite{LY2}, since our equation is degenerate, we need to exploit the nondivergence form of the normalized equation. Consequently, the argument of flattening the boundaries becomes much more intricate, and unlike in \cite{LY2}, where the De Giorgi-Nash-Moser Harnack inequality is applied, we use the Krylov-Safonov Harnack inequality for nondivergence form equations. Furthermore, there are additional first-order terms that require control over the size of the coefficients.

To prove Theorem \ref{Theorem_improved}, for $\eta > 0$, we consider the approximating equation
\begin{equation}\label{main_problem_approximate}
\left\{
\begin{aligned}
-\dv \Big( (\eta + |D u_\eta|^2)^{\frac{p-2}{2}} D u_\eta \Big) &=0 \quad \mbox{in }\widetilde{\Omega},\\
\frac{\partial u_\eta}{\partial \nu} &= 0 \quad \mbox{on}~\partial{D}_{i},~i=1,2,\\
 u_\eta &= \varphi \quad \mbox{on } \partial \Omega.
\end{aligned}
\right.
\end{equation}
Since $\|u_\eta\|_{C^{1,\alpha}(\Omega_{1})}$ is bounded independent of $\eta$, it suffices to prove \eqref{gradient_-1/2+beta} for $u_\eta$. Therefore, we will focus on \eqref{main_problem_approximate} throughout the rest of this section, and denote $u = u_\eta$ for simplicity. Note that $u$ satisfies the normalized $p$-Laplace equation
$$
a^{ij} D_{ij} u = 0 \quad \mbox{in }\Omega_{1},
$$
where
\begin{equation}\label{a_ij_2}
a^{ij} = \delta_{ij} + (p-2)(\eta +|D u|^2)^{-1} D_i u D_j u
\end{equation}
satisfies
\begin{align}\label{a_elliptic_2}
(p-1) |\xi|^2 \le a^{ij}\xi_i \xi_j \le |\xi|^2, \quad \forall \xi \in \bR^n \quad &\mbox{when}~~1<p<2,\nonumber\\
|\xi|^2 \le a^{ij}\xi_i \xi_j \le (p-1) |\xi|^2, \quad \forall \xi \in \bR^n \quad &\mbox{when}~~p \ge 2.
\end{align}

For a small $r_0$ independent of $\varepsilon$, we only need to show \eqref{gradient_-1/2+beta} in $\Omega_{r_0}$, as $|D u|$ is bounded in $\Omega_{1/2} \setminus \Omega_{r_0}$ independent of $\varepsilon$. For any $x \in \Omega_{r_0}$, we estimate $|D u(x)|$ as follows: First we consider the equation in $\Omega_{2r} \setminus \Omega_{r/4}$ for $r \in (\sqrt{\varepsilon}, r_0]$, we perform a suitable change of variables that maps the domain to a flat ``annular cylinder". After the change of variables, $u$ will satisfy a second-order uniformly elliptic equation in non-divergence form, and the Neumann boundary condition on the upper and lower boundaries of the domain. Then we obtain a Harnack inequality through the Krylov-Safonov theorem. Together with the maximum principle, this gives the oscillation of $u$ in $\Omega_{r}$ for $r \in (\sqrt{\varepsilon}, r_0]$  with a decay rate $r^{2\beta}$ for some positive $\varepsilon$-independent $\beta$. Then the desired estimate on $|D u(x)|$ follows from the decay rate of $\osc_{2(\varepsilon + |x'|^2)^{1/2}} u$ and Theorem \ref{thm-1/2}.

Let $r \in (\sqrt{\varepsilon}, r_0]$, where $r_0$ is an $\varepsilon$-independent constant to be determined later. We define
$$
\tilde{h}_i(x') :=
\left\{
\begin{aligned}
&h_i(x') && \mbox{when}~~|x'| \le 2r_0,\\
&0 &&  \mbox{when}~~|x'| > 2r_0
\end{aligned}
\right.
$$
for $i=1,2$. We denote
$$
Q_{s,t} := \{ y = (y', y_n) \in \bR^n ~\big|~ |y'| < s, |y_n| < t \},
$$
and for $y \in Q_{2r,r^2} \setminus Q_{r/4, r^2}$, we define the map $x = \Phi(y)$ by
\begin{equation}\label{Phi}
\left\{
\begin{aligned}
x' &= y' - g(y),\\
x_n &= \frac{1}{2} \Big[\frac{y_n}{r^2} (\varepsilon + \tilde h_1(y') - \tilde h_2(y')) + \tilde h_1(y') +\tilde h_2(y') \Big],
\end{aligned}
\right.
\end{equation}
where
\begin{equation}\label{def_g}
g(y) = (y_n - r^2)(y_n + r^2)(\Theta y_n + \Xi),
\end{equation}
$$
\left\{
\begin{aligned}
\Theta &= \frac{1}{8r^6} [\varepsilon + \tilde h_1(y') - \tilde h_2(y')] D_{y'} [\tilde h_1^\mu(y') + \tilde h_2^\mu(y')],\\
\Xi &= \frac{1}{8r^4} [\varepsilon + \tilde h_1(y') - \tilde h_2(y')] D_{y'} [\tilde h_1^\mu(y') - \tilde h_2^\mu(y')],
\end{aligned}
\right.
$$
$\tilde h_i^\mu$ is a mollification of $\tilde h_i$ given by
\begin{equation}\label{mollification}
\tilde h_i^\mu (y') := \int_{\bR^{n-1}} \tilde h_i (y' - \mu z') \varphi(z') \, dz',
\end{equation}
$\varphi$ is a positive smooth function with unit integral supported in $B_1 \subset \bR^{n-1}$, and
$$\mu = \frac{r^4 - y_n^2}{r}\ge 0.$$

Here we briefly explain the motivation for defining the map $\Phi$ as above: to ensure that $y' = x'$ on $\{y_n = \pm r^2\}$, which is $g\big|_{y_n = \pm r^2} = 0$, we setup the ansatz \eqref{def_g} for $g$. Next, we want the function $v(y):= u(\Phi(y))$ to satisfy the Neumann boundary condition on $\{y_n = \pm r^2\}$, which leads to (see details in Lemma \ref{Change_of_variable_lemma})
$$
\left\{
\begin{aligned}
&\Big(-D_{y_n}g , \frac{1}{2r^2} (\varepsilon + \tilde h_1(y') - \tilde h_2(y'))  \Big) \parallel (-D_{x'} \tilde{h}_1 ,1) \quad \mbox{on}~~\{y_n = r^2\},\\
&\Big(-D_{y_n}g , \frac{1}{2r^2} (\varepsilon + \tilde h_1(y') - \tilde h_2(y'))  \Big) \parallel (-D_{x'} \tilde{h}_2 ,1) \quad \mbox{on}~~\{y_n = -r^2\}.
\end{aligned}
\right.
$$
Using the ansatz \eqref{def_g} and solving for $\Theta$ and $\Xi$, we have
$$
\left\{
\begin{aligned}
\Theta &= \frac{1}{8r^6} [\varepsilon + \tilde h_1(y') - \tilde h_2(y')] D_{y'} [\tilde h_1(y') + \tilde h_2(y')],\\
\Xi &= \frac{1}{8r^4} [\varepsilon + \tilde h_1(y') - \tilde h_2(y')] D_{y'} [\tilde h_1(y') - \tilde h_2(y')].
\end{aligned}
\right.
$$
Note that the equation of $v$ involves second-order derivatives of $\Phi$, and hence involves third-order derivatives of $\tilde{h}_1$ and $\tilde{h}_2$. However, $\tilde{h}_1$ and $\tilde{h}_2$ are only $C^{1,1}$,  so we introduce the mollification \eqref{mollification} to overcome the lack of regularities. Here $\mu$ is chosen so that $\tilde h_i^\mu = \tilde h_i$ on $\{y_n = \pm r^2\}$, and the coefficients of the equation of $v$ have the desired estimates.

Throughout this section, unless specify otherwise, we use $C$ to denote positive constants that could be different from line to line, and depend only on $n$, $p$, $c_1$, and $c_2$, where $c_1$ and $c_2$ are defined in \eqref{fg_1} and \eqref{def:c_2}, respectively.

\begin{lemma}\label{Change_of_variable_lemma}
There exists an $r_0 > 0$ independent of $\varepsilon$, such that when $r \in (\sqrt{\varepsilon}, r_0]$ and $\Phi$ is given as \eqref{Phi}, then:
\begin{enumerate}
\item There exists a positive constant $C$ independent of $\varepsilon$ and $r$, such that
$$\frac{I}{C} \le D \Phi(y) \le C I, \quad y \in Q_{2r,r^2} \setminus Q_{r/4, r^2},$$
and hence $\Phi$ is invertible.
\item $$Q_{1.9r, r^2} \setminus Q_{0.35r, r^2} \subset \Phi^{-1}(\Omega_{ 2r} \setminus \Omega_{r/4}),$$
and
$$\Omega_{r} \setminus \Omega_{ r/2} \subset \Phi (Q_{1.1r, r^2} \setminus Q_{0.4r, r^2}).$$
\item Let $u \in W^{1,p}(\Omega_{ 2r} \setminus \Omega_{r/4})$ be a solution of
\begin{equation}\label{local_equation}
\left\{
\begin{aligned}
-\dv \Big( (\eta + |D u|^2)^{\frac{p-2}{2}} D u \Big) &=0 \quad \mbox{in }\Omega_{2r} \setminus \Omega_{r/4},\\
\frac{\partial u}{\partial \nu} &= 0 \quad \mbox{on } (\Gamma_+ \cup \Gamma_-) \cap \overline{\Omega_{2r} \setminus \Omega_{r/4}}\\
\end{aligned}
\right.
\end{equation}
for some $\eta > 0$, and $v(y) = u(\Phi(y))$. Then $v$ satisfies an elliptic equation
\begin{equation}\label{v_equation}
\left\{
\begin{aligned}
\tilde{a}^{ij} D_{ij}v(y) + \tilde{b}_i D_i v(y) &=0 \quad \mbox{in }Q_{1.9r, r^2} \setminus Q_{0.35r, r^2},\\
\frac{\partial v}{\partial \nu}(y) &= 0 \quad \mbox{on } \{y_n = \pm r^2\},\\
\end{aligned}
\right.
\end{equation}
with
$$
\frac{I}{C} \le \tilde{a} \le C I, \quad |\tilde{b}| \le \frac{C}{r}.
$$
\end{enumerate}
\end{lemma}

\begin{proof}
By \eqref{fg_0}, we have
\begin{equation}\label{h_derivatives}
|D_{y'}^k \tilde{h}_{i}(y') | \le Cr^{2-k}, \quad |D_{y'}^k \tilde{h}^\mu_{i}(y') | \le Cr^{2-k} \quad \mbox{for}~~i = 1,2~~\mbox{and}~~k =0,1,2.
\end{equation}
Therefore, when $y \in Q_{2r,r^2} \setminus Q_{r/4, r^2}$,
$$
|D_{y'} x' - I_{(n-1)\times(n-1)}| = |D_{y'} g(y)| \le Cr^2.
$$
\begin{align*}
D_{y_n} x' =& -D_{y_n} g(y)\\
=& -2y_n(\Theta y_n + \Xi) - (y_n^2 - r^4)[(D_{y_n} \Theta) y_n + \Theta + D_{y_n}\Xi],
\end{align*}
and
$$
|\Theta| \le Cr^{-3}, \quad |\Xi| \le Cr^{-1}.
$$
By \eqref{mollification}, we have for $y \in Q_{2r,r^2} \setminus Q_{r/4, r^2}$,
\begin{align*}
|D_{y_n} D_{y'} \tilde{h}_i^\mu (y')| =& \left| - \frac{\partial \mu}{\partial y_n} \int_{\bR^{n-1}}D^2_{y'} \tilde{h}_i(y' - \mu z')z' \varphi(z') \, dz'\right| \\
=& \left| \frac{2y_n}{r} \int_{\bR^{n-1}}D^2_{y'} \tilde{h}_i(y' - \mu z')z' \varphi(z') \, dz' \right| \le Cr \quad \mbox{for}~~i=1,2.
\end{align*}
Therefore,
$$
|D_{y_n}\Theta| \le Cr^{-3}, \quad |D_{y_n} \Xi| \le C r^{-1},
$$
and hence
\begin{align*}
|D_{y_n} x'| \le& |2y_n(\Theta y_n + \Xi)| + |(y_n^2 - r^4)[(D_{y_n} \Theta) y_n + \Theta + D_{y_n}\Xi]|\\
\le& Cr + Cr^4 [r^{-1} + r^{-3} + r^{-1}] \le Cr.
\end{align*}
By \eqref{h_derivatives},
\begin{align*}
&|D_{y'}x_n| =  \frac{1}{2} \Big|\frac{y_n}{r^2} (D_{y'}\tilde h_1(y') - D_{y'}\tilde h_2(y')) + D_{y'}\tilde h_1(y') +D_{y'}\tilde h_2(y' ) \Big| \le Cr.
\end{align*}
And lastly,
\begin{align*}
D_{y_n} x_n = \frac{1}{2r^2} (\varepsilon + \tilde h_1(y') - \tilde h_2(y')).
\end{align*}
By \eqref{fg_1},
$$
\frac{1}{C} \le D_{y_n} x_n \le C.
$$
Then (a) follows by shrinking $r_0$ to be sufficiently small.

Since $g(y) = 0$ when $y = \pm r^2$, $\Phi$ maps the upper and lower boundaries of $Q_{2r, r^2} \setminus Q_{r/4, r^2}$ onto the upper and lower boundaries of $\Omega_{2r} \setminus \Omega_{r/4}$, respectively. Then (b) simply follows from the fact that $|g(y)| \le Cr^3$, and we can shrink $r_0$ so that $|g(y)| \le r/10$.

To verify (c), note that $u$ is smooth from the classical elliptic theory. We compute by the chain rule,
\begin{align*}
D_{x_k} u(x) =& D_{y_i} v(y) D_{x_k} y_i,\\
D_{x_k x_l} u(x) =&  D_{y_i}  D_{y_j } v(y) D_{x_k} y_i D_{x_l} y_j + D_{y_i} v(y) D_{x_k x_l} y_i.
\end{align*}
Recall that $u(x)$ satisfies the equation
$$
a^{kl} D_{x_k x_l} u(x) =0,
$$
where the matrix $a$ is given by \eqref{a_ij_2}. If we define
$$
\tilde{a}^{ij} := a^{kl} D_{x_k} y_i D_{x_l} y_j, \quad \tilde{b}^i := a^{kl} D_{x_k x_l} y_i,
$$
then $v(y)$ satisfies
$$
\tilde{a}^{ij} D_{ij}v(y) + \tilde{b}_i D_i v(y) =0.
$$
Next, we show that $v$ satisfies the Neumann boundary condition on $\{y_n = \pm r^2\}$. We will show the boundary condition on $\{y_n = r^2\}$, as the one on $\{y_n = -r^2\}$ follows similarly. By the chain rule,
$$
D_{y} v(y) \cdot e_n = D_x u(x)\cdot  D_y \Phi e_n,
$$
where $e_n:= (0,\ldots,0,1)$. Therefore, it suffices to show that
\begin{equation}\label{parallel}
 D_y \Phi e_n = \Big(-D_{y_n}g , \frac{1}{2r^2} (\varepsilon + \tilde h_1(y') - \tilde h_2(y'))  \Big) \parallel (-D_{x'} \tilde{h}_1 ,1) \quad \mbox{on}~~\{y_n = r^2\}.
\end{equation}
Note that when $y_n = r^2$, we have $g(y) = 0$, $y' = x'$, $\mu = 0$, and $\tilde{h}_1^\mu = \tilde{h}_1$. Therefore,
\begin{align*}
D_{y_n}g =& (y_n + r^2)(\Theta y_n + \Xi)\Big|_{y_n = r^2}\\
=& \frac{1}{2r^2} (\varepsilon + \tilde h_1(y') - \tilde h_2(y')) D_{y'} \tilde{h}_1^\mu (y')\\
=& \frac{1}{2r^2} (\varepsilon + \tilde h_1(y') - \tilde h_2(y')) D_{x'} \tilde{h}_1 (x').
\end{align*}
This implies \eqref{parallel}.

Finally, we show that the coefficients $\tilde{a}$ and $\tilde{b}$ satisfy the desired estimates. From part (a), we know that
$$
\frac{I}{C} \le D_x y = D_x \Phi^{-1}(x) \le C I,
$$
which together with \eqref{a_elliptic_2} implies that
$$
\frac{I}{C} \le \tilde{a} \le C I.
$$
To estimate $\tilde{b}$, we differentiate $\partial y_i/ \partial x_k \cdot \partial x_k / \partial y_j = \delta_{ij}$ in $x_l$. Note that by chain rule, we have
$$
\frac{\partial^2 y_i}{\partial x_k \partial x_l} \frac{\partial x_k}{\partial y_j} + \frac{\partial y_i}{\partial x_k} \frac{\partial y_m}{\partial x_l} \frac{\partial^2 x_k}{\partial y_j \partial y_m} = 0.
$$
Since $I/C \le D_x y \le CI$ and $I/C \le D_y x\le CI$, it suffices to estimate $D^2_y x$, which is $D^2_y \Phi(y)$. It is easy to see that
$$
\left| \frac{\partial^2 x_n}{\partial y^2} \right| \le \frac{C}{r}.
$$
To estimate $\partial^2 x'/ \partial y^2$, the key terms are
$$
D_{y'}^3 \tilde{h}_i^\mu (y'),\quad
D_{y_n}D_{y'}^2 \tilde{h}_i^\mu (y'),\quad
D_{y_n}^2D_{y'} \tilde{h}_i^\mu (y'),\quad i=1,2.
$$
By \eqref{mollification} and integration by parts, we have
\begin{align*}
D_{y'} \tilde{h}_i^\mu (y') =& \int_{\bR^{n-1}} D_{y'} \tilde h_i (y'  - \mu z') \varphi(z') \, dz' \\
=& - \frac{1}{\mu} \int_{\bR^{n-1}} D_{z'} \tilde h_i (y'- \mu z') \varphi(z') \, dz'\\
=& \frac{1}{\mu} \int_{\bR^{n-1}}  \tilde h_i (y' - \mu z') D_{z'}\varphi(z') \, dz'.
\end{align*}
Then
$$
|D_{y'}^3 \tilde{h}_i^\mu (y')| \le \frac{C \|h_i\|_{C^{1,1}}}{\mu} \le \frac{Cr}{r^4 - y_n^2}.
$$
Similarly,
\begin{align}\label{mollification_yn_derivative}
D_{y_n} \tilde{h}_i^\mu (y') =& \frac{2y_n}{r} \int_{\bR^{n-1}} D_{y'} \tilde h_i (y' - \mu z') \cdot z' \varphi(z') \, dz' \nonumber\\
=& \frac{2y_n}{\mu r} \int_{\bR^{n-1}} \tilde h_i (y' - \mu z') D_{z'} \cdot (z' \varphi(z')) \, dz',
\end{align}
so
\begin{align*}
|D_{y_n}D_{y'}^2 \tilde{h}_i^\mu (y')| = \left| \frac{2y_n}{\mu r} \int_{\bR^{n-1}} D_{y'}^2 \tilde h_i (y' - \mu z') D_{z'} \cdot (z' \varphi(z')) \, dz' \right| \le \frac{Cr^2}{r^4 - y_n^2}.
\end{align*}
Differentiating the first line of \eqref{mollification_yn_derivative} in $y_n$, we have
\begin{align*}
D_{y_n}^2 \tilde{h}_i^\mu (y' ) =& \frac{2}{r} \int_{\bR^{n-1}} \sum_{k=1}^{n-1} D_{y_k} \tilde h_i (y' - \mu z') z_k \varphi(z') \, dz'\\
&+ \frac{4y_n^2}{ r^2} \int_{\bR^{n-1}} \sum_{k,l=1}^{n-1} D_{y_k y_l} \tilde h_i  (y'  - \mu z')  z_k z_l \varphi(z') \, dz'\\
 =& \frac{2}{r} \int_{\bR^{n-1}} \sum_{k=1}^{n-1} D_{y_k} \tilde h_i (y' - \mu z') z_k \varphi(z') \, dz'\\
&+ \frac{4y_n^2}{\mu r^2} \int_{\bR^{n-1}} \sum_{k,l=1}^{n-1} D_{y_l} \tilde h_i  (y'  - \mu z') D_{z_k} \big(z_k z_l \varphi(z')\big) \, dz'.
\end{align*}
Therefore,
\begin{align*}
|D_{y_n}^2 D_{y'} \tilde{h}_i^\mu (y')| \le& \frac{2}{r} \int_{\bR^{n-1}} |D_{y'}^2 \tilde h_i (y' - \mu z')| |z' \varphi(z')| \, dz'\\
&+ \frac{4y_n^2}{\mu r^2} \int_{\bR^{n-1}} |D_{y'}^2 \tilde h_i  (y'  - \mu z')| | D_{z'} (z'\otimes z' \varphi(z'))| \, dz'\\
\le& \frac{Cr^3}{r^4 - y_n^2}.
\end{align*}
By these estimates above and straightforward computations, we have
$$
|D_{y}^2 \Phi(y)| \le Cr^{-1},
$$
which implies $|\tilde{b}| \le Cr^{-1}$.
\end{proof}

\begin{lemma}\label{harnack_inequality_lemma}
Let $r_0$ be as in Lemma \ref{Change_of_variable_lemma}, and let $r \in (\sqrt{\varepsilon}, r_0]$. If $u \in W^{1,p}(\Omega_{2r} \setminus \Omega_{r/4})$ is a nonnegative solution of \eqref{local_equation} for some $\eta > 0$,
then,
\begin{equation}\label{harnack}
\sup_{\Omega_{r} \setminus \Omega_{r/2}} u \le C \inf_{\Omega_{r} \setminus \Omega_{r/2}} u,
\end{equation}
for some constant $C >0$ depending only on $n$, $p$, $c_1$, and $c_2$, but independent of $\varepsilon$, $\eta$, $r$, and $u$.
\end{lemma}

\begin{proof}
We take the change of variable $y = \Phi^{-1}(x)$, where $\Phi$ is given as \eqref{Phi}. Let $v(y) = u(x)$. By Lemma \ref{Change_of_variable_lemma} (c), $v$ satisfies the equation \eqref{v_equation}.

For $i,j = 1,2,\ldots, n-1$, we take the even extension of $\tilde{a}^{ij}$, $\tilde{a}^{nn}$, $\tilde{b}^i$, and $v$ with respect to $y_n = r^2$, and take odd extension of $\tilde{a}^{in}$, $\tilde{a}^{ni}$, and $\tilde{b}^n$ with respect to $y_n = r^2$. Then we take the periodic extension (so that the period is equal to $4r^2$). We still denote them by $\tilde{a}$, $\tilde{b}$, and $v$ after the extension. Then $v$ satisfies
$$
\tilde{a}^{ij} D_{ij}v(y) + \tilde{b}_i D_i v(y) =0 \quad \mbox{in }Q_{1.9r, 2r} \setminus Q_{0.35r, 2r}.
$$
Setting $\bar{a}^{ij}(y) = \tilde{a}^{ij}(ry)$, $\bar{b}^{i}(y) = r \tilde{b}^{i}(ry)$, and $\bar{v}(y) = v(ry)$, we see that $\bar{v}$ satisfies
$$
\bar{a}^{ij} D_{ij}\bar v(y) + \bar{b}_i D_i \bar v(y) =0 \quad \mbox{in }Q_{1.9, 2} \setminus Q_{0.35, 2},
$$
with
$$
\frac{I}{C} \le \bar{a} \le C I, \quad |\bar{b}| \le C.
$$
Since $Q_{1.9, 2} \setminus Q_{0.35, 2}$ is connected when $n \ge 3$, by the Krylov-Safonov theorem (see Section 4.2 of \cite{Krylov}), we have
$$
\sup_{Q_{1.1, 1} \setminus Q_{0.4, 1}} \bar{v} \le C \inf_{Q_{1.1, 1} \setminus Q_{0.4, 1}} \bar{v}.
$$
This implies
$$
\sup_{Q_{1.1r, r^2} \setminus Q_{0.4r, r^2}} v \le C \inf_{Q_{1.1r, r^2} \setminus Q_{0.4r, r^2}} v.
$$
Finally, \eqref{harnack} follows by reverting the changes of variables and Lemma \ref{Change_of_variable_lemma} (b).
\end{proof}

The following estimate on the oscillation of $u$ is a direct consequence of Lemma \ref{harnack_inequality_lemma}.

\begin{corollary}\label{coro_3d}
For $n \ge 3$, let $u \in W^{1,p}(\Omega_{1})$ be a solution of \eqref{main_problem_approximate} for some $\eta > 0$. Then there exist positive constants $C$ and $\beta$, depending only on $n, p, c_1,$ and $c_2$, such that
\begin{equation}\label{oscillation}
\underset{\Omega_{r}}{\osc}~u \le C r^{\beta} \underset{\Omega_{1}}{\osc}~u, \quad \forall\, r \in (\sqrt{\varepsilon}, 1/2).
\end{equation}
\end{corollary}

\begin{proof}
It suffices to prove \eqref{oscillation} for $r \in (\sqrt{\varepsilon}, r_0]$, where $r_0$ is the same as in Lemma \ref{Change_of_variable_lemma}. Let $\sqrt{\varepsilon} < r \le r_0$ and $v = u - \inf_{\Omega_{2r}} u$. Then $v \ge 0$ in $\Omega_{2r}$. By Lemma \ref{harnack_inequality_lemma}, we have
\begin{align*}
\sup_{\Omega_r \setminus \Omega_{r/2}} v \le C_1 \inf_{\Omega_r \setminus \Omega_{r/2}} v,
\end{align*}
where $C_1 > 1$ is a constant independent of $r$. Since $v$ satisfies equation \eqref{main_problem_narrow}, by the maximum principle,
\begin{align*}
\sup_{\Omega_r \setminus \Omega_{r/2}} v = \sup_{\Omega_r} v, \quad \inf_{\Omega_r \setminus \Omega_{r/2}} v = \inf_{\Omega_r} v.
\end{align*}
Therefore,
\begin{align*}
\sup_{\Omega_r} v \le C_1 \inf_{\Omega_r} v,
\end{align*}
which implies
$$
\sup_{\Omega_r} u \le C_1 \inf_{\Omega_{r}} u - (C_1 -1) \inf_{\Omega_{2r}} u.
$$
Adding the above inequality with
$$
(C_1 - 1) \sup_{\Omega_r} u \le (C_1 - 1) \sup_{\Omega_{2r}} u,
$$
and dividing both sides by $C_1$, we have
$$\underset{\Omega_{r}}{\osc}~u \le \frac{C_1 - 1}{C_1} \underset{\Omega_{2r}}{\osc}~u.$$
Finally, \eqref{oscillation} follows from iterating the inequality above.
\end{proof}

Now we are ready to prove Theorem \ref{Theorem_improved}.
\begin{proof}[Proof of Theorem \ref{Theorem_improved}]
It suffices to show \eqref{gradient_-1/2+beta} for $x \in \Omega_{1/8}$ and $\varepsilon\in(0,1/32)$. By Corollary \ref{coro_3d}, there exist positive constants $C$ and $\beta$, depending only on $n, p, c_1,$ and $c_2$, such that
$$
\underset{\Omega_{8\eta}}{\osc}~u \le C (\varepsilon + |x'|^2)^{\beta} \underset{\Omega_{1}}{\osc}~u,
$$
where $\eta=\frac{1}{4}(\varepsilon+|x'|^2)^{\frac{1}{2}}$. Then by Theorem \ref{thm-1/2}, we have
\begin{align*}
|D u(x)| \le & C (\varepsilon+|x'|^2)^{-\frac{1}{2}}\underset{\Omega_{x,\eta}}{\osc}~u\\ \le & C (\varepsilon+|x'|^2)^{-\frac{1}{2}} \underset{\Omega_{8\eta}}{\osc}~u \\
\le & C (\varepsilon + |x'|^2)^{-\frac{1}{2} + \beta} \underset{\Omega_{1}}{\osc}~u.
\end{align*}
The theorem is proved.
\end{proof}

\section{The \texorpdfstring{$p > n + 1$}{p > n + 1} case}\label{sec4}
In this section, we establish a more explicit gradient estimate for the equation  \eqref{main_problem_narrow} when $p > n + 1$, with a blow-up rate of order $\varepsilon^{-\alpha}$ for any $\alpha > \frac{n}{2(p-1)}$. Throughout this section, in addition to \eqref{fg_0} and \eqref{fg_1}, we need to further assume that $h_1$ and $h_2$ are $C^{2}$ strictly convex and strictly concave functions respectively, satisfying \eqref{fg_convex}. Let $\nu$ denote the normal vector on $\Gamma_{\pm}$, pointing upwards and downwards respectively.

To obtain the improved gradient estimate, in the following lemma, we construct a supersolution to show that the oscillation of $u$ enjoys a better decay rate. Then the desired gradient estimate  \eqref{gradient_2d} follows by using Theorem \ref{thm-1/2}.
\begin{lemma}\label{supersolution_lemma}
Let $n\geq 2$, $p > n+1$, $\Gamma_+, \Gamma_-, h_1, h_2$ be as above. For any $\delta \in(0, p-n-1)$, let $v(x) = (|x'|^2  + (2+\delta)x_n^2)^{\gamma/2}$. Then for any $\gamma \in(0,\frac{p-n-1-\delta}{p-1})$, there exists a constant $\mu\in(0,1/2)$ depending only on $n$, $p$, $\delta$, $\gamma$,  $\kappa_1$, $\kappa_2$, and the modulus of continuity for $D^2h_1(x')$ and $D^2h_2(x')$ at $x' = 0$,  such that for any  $\varepsilon\in(0,\mu^2/\kappa_2)$,
$$
\left\{
\begin{aligned}
-\dv (|D v|^{p-2} D v) &> 0 \quad \mbox{in }\Omega_{\mu/\kappa_2}\setminus \Omega_{\varepsilon/\mu} ,\\
\frac{\partial v}{\partial \nu} &> 0 \quad \mbox{on } (\Gamma_+ \cup \Gamma_-) \cap \overline{\Omega}_{\mu/\kappa_2}.\\
\end{aligned}
\right.
$$
\end{lemma}

\begin{proof}
We denote $R(x) = (|x'|^2 + (2+\delta)x_n^2)^{1/2}$, so that $v(x) = R(x)^\gamma$.
Using Taylor expansion up to order 2, from \eqref{fg_0} we know that for any $|x'|<1$.
\begin{align}\label{taylor}
    h_1(x')=\big\langle \int_{0}^1 (1-t) D^2h_1(tx')dt \cdot x',x'\big\rangle,  \quad h_2(x')=\big\langle \int_{0}^1 (1-t) D^2h_2(tx')dt \cdot x',x'\big\rangle.
\end{align}
By \eqref{taylor} and \eqref{fg_convex}, we have
\begin{equation}\label{fg_l_u}
 \frac{\kappa_1}{2}|x'|^2 \le  h_1(x') \le \frac{\kappa_2}{2}|x'|^2, \quad  \frac{\kappa_1}{2}|x'|^2 \le  h_2(x') \le \frac{\kappa_2}{2}|x'|^2  \quad\mbox{for}~~|x'|< 1.
\end{equation}
We also note that for any $|x'|<1$,
\begin{align}\label{x.D}
    \big\langle Dh_1(x'), x'\big\rangle =\big\langle\int_{0}^1 \frac{d}{dt} Dh_1(tx')dt,x'\big\rangle=\big\langle\int_{0}^1 D^2h_1(tx')dt\cdot x',x'\big\rangle,
\end{align}
and similarly
\begin{align*}
    \big\langle Dh_2(x'), x'\big\rangle =\big\langle\int_{0}^1 D^2h_2(tx')dt\cdot x',x'\big\rangle.
\end{align*}
Since $h_1$ and $h_2$ are $C^2$, for any $\delta\in(0,p-n-1)$, there is a sufficiently small $r_0\in(0,1/2)$ depending only on $n$, $\delta$, $\kappa_1$, and the modulus of continuity for $D^2h_1(x')$ and $D^2h_2(x')$ at $x' = 0$, such that
\begin{align}\label{D2-cont}
    |D^2h_1(x')-D^2h_1(0)|\leq\frac{\kappa_1 \delta}{8+2\delta},\quad |D^2h_2(x')-D^2h_2(0)|\leq \frac{\kappa_1 \delta}{8+2\delta} \quad\mbox{for}~~|x'|\leq r_0.
\end{align}
Thus by \eqref{taylor}, \eqref{x.D}, \eqref{D2-cont},  \eqref{fg_convex}, and the triangle inequality, we obtain
\begin{equation}
\begin{aligned}\label{tay-x.D}
    &(2+\delta)h_1(x')-\big\langle Dh_1(x'), x'\big\rangle\\&=\big\langle \int_{0}^1 (2+\delta) (1-t) D^2h_1(tx')dt \cdot x',x'\big\rangle-\big\langle\int_{0}^1 D^2h_1(tx')dt\cdot x',x'\big\rangle\\
    &\geq \big\langle \int_{0}^1 (2+\delta)(1-t) D^2h_1(0)dt \cdot x',x'\big\rangle-\frac{2+\delta}{2}\frac{\kappa_1\delta}{8+2\delta}|x'|^2\\
    &\quad -\big\langle\int_{0}^1 D^2h_1(0)dt\cdot x',x'\big\rangle - \frac{\kappa_1\delta}{8+2\delta}|x'|^2\\
    &=\frac{\delta}{2}\langle D^2h_1(0)\cdot x',x'\rangle- \frac{\kappa_1\delta}{4} |x'|^2\geq \frac{\kappa_1\delta}{4} |x'|^2\geq 0.
\end{aligned}
\end{equation}
By direct computations, we have
\begin{align*}
D v &= \gamma R^{\gamma-2} (x', (2+\delta)x_n),\\
D_{ii} v&= \gamma R^{\gamma-2} + \gamma(\gamma-2)R^{\gamma-4}x_i^2, \quad \text{for}\; i\in\{1,2,\ldots,n-1\},\\
D_{ij}v &= \gamma(\gamma-2)R^{\gamma-4}x_ix_j, \quad \text{for}\; i\neq j, \; i,j\in\{1,2,\ldots,n-1\},\\
D_{in}v &= \gamma(\gamma-2)R^{\gamma-4}(2+\delta)x_i x_n , \quad \text{for}\; i\in\{1,2,\ldots,n-1\},\\
D_{nn}v &= \gamma R^{\gamma-2}(2+\delta) + \gamma(\gamma-2)R^{\gamma-4}(2+\delta)^2x_n^2.
\end{align*}
On $\Gamma_+  \cap \overline{\Omega}_{r_0}$,
$$
\nu = \frac{1}{\sqrt{1+|Dh_1(x')|^2}} (-Dh_1(x'),1).
$$
Then by \eqref{tay-x.D},
\begin{align*}
\frac{\partial v}{\partial \nu}=& \frac{\gamma R^{\gamma-2}}{\sqrt{1+(Dh_1(x'))^2}} \Big[ -\langle Dh_1(x'), x'\rangle  + (2 + \delta)(\frac{\varepsilon}{2} + h_1(x')) \Big]\\
\ge &   \frac{\gamma R^{\gamma-2}}{\sqrt{1+(Dh_1(x'))^2}} \Big[ \frac{\kappa_1\delta}{4}  |x'|^2 + (1+ \frac{\delta}{2}) \varepsilon \Big]> 0.
\end{align*}
A similar computation shows that $\frac{\partial v}{\partial \nu} > 0$ on $\Gamma_-  \cap \overline{\Omega}_{r_0}$.

Next, we compute in $\Omega_1\setminus\{0\}$,
\begin{align*}
&\dv (|D v|^{p-2} D v) |D v|^{4-p} = |D v|^2 \Delta v + (p-2)D_i v D_j v D_{ij} v\\
&= \gamma^2 R^{2\gamma - 4} \Big( |x'|^2  + (2+\delta)^2x_n^2 \Big)\Big( (n+1+\delta)\gamma R^{\gamma-2} \\
&\quad\quad + \gamma(\gamma-2)R^{\gamma-4}(|x'|^2 + (2+\delta)^2x_n^2)\Big)\\
&\quad + (p-2) \gamma^2 R^{2\gamma - 4} |x'| ^2 \Big( \gamma R^{\gamma-2} + \gamma(\gamma-2)R^{\gamma-4}|x'|^2 \Big)\\
&\quad +  2(p-2)(2+\delta)^2\gamma^2 R^{2\gamma - 4}\gamma(\gamma-2)R^{\gamma-4}|x'|^2 x_n^2\\
&\quad + (p-2) \gamma^2 R^{2\gamma - 4} (2+\delta)^2x_n^2 \Big( \gamma R^{\gamma-2}(2+\delta) + \gamma(\gamma-2)R^{\gamma-4}(2+\delta)^2x_n^2 \Big)\\
&=\gamma^3 R^{3\gamma-8} \Big\{|x'|^4\big[n+\delta+(p-1)(\gamma-1)\big]\\&\qquad+|x'|^2x_n^2 \big[(2+\delta)(n+\delta+p-1)\\
&\quad\quad\quad+(2+\delta)^2(n+1+\delta+(p-2)(2+\delta)+2(\gamma-2)(p-1))\big]\\
&\quad\quad
+x_n^4\big[(2+\delta)^3(n+1+\delta+(p-2)(2+\delta)+(\gamma-2)(p-1)(2+\delta))\big]\Big\}.
\end{align*}
Thus $\dv (|D v|^{p-2} D v) |D v|^{4-p} \gamma^{-3} R^{8-3\gamma}$ is a 4th order homogeneous polynomial of $|x'|$ and $x_n$.
Since $\gamma \in(0,\frac{p-n-1-\delta}{p-1})$,
we have
$$
n+\delta+(p-1)(\gamma-1)< 0.
$$
Therefore there exists a sufficiently small $\mu_0\in(0,1/2)$ depending only on $n$, $p$, $\gamma$, and  $\delta$, such that if $|x_n|\leq \mu_0 |x'|$ and $x\neq 0$, we have
$$\dv (|D v|^{p-2} D v)>0.$$
We then take
$ \mu=\min\{\mu_0, \kappa_2 r_0\}$, so that $\mu/\kappa_2 \leq r_0$ and thus $\frac{\partial v}{\partial \nu} > 0 \quad \mbox{on } (\Gamma_+ \cup \Gamma_-) \cap \overline{\Omega}_{\mu/\kappa_2}$.
 Note that when $x\in \Omega_{\mu/\kappa_2}\setminus \Omega_{\varepsilon/\mu}$,
 by \eqref{fg_l_u} we have
 $$
 |x_n|\leq \frac{\varepsilon}{2}+\frac{\kappa_2}{2} |x'|^2\leq \mu |x'|\leq \mu_0 |x'|.
 $$
 This concludes the proof.
\end{proof}

\begin{proof}[Proof of Theorem \ref{thm:2d}]
It suffices to show \eqref{gradient_2d} for $\delta\in(0, (p-n-1)/2)$, $x \in \Omega_{\mu/2\kappa_2}$, and $\varepsilon\in(0,\mu^2/(\kappa_2+\kappa_2^2))$, where $\mu\in(0,1/2)$ is defined in Lemma \ref{supersolution_lemma} with $\gamma = \frac{p-1-n-2\delta}{p-1}\in(0,1)$. Without loss of generality, we may assume that $u(0) = 0$ and ${\osc_{\Omega_1}}u = 1$. By Theorem  \ref{thm-1/2},
\begin{equation}\label{less-eps}
|u(x))| \le C \sqrt{\varepsilon} \quad \mbox{for}~~x\in \overline{\Omega}_{\varepsilon/\mu}.
\end{equation}
Let $v$ be the function defined in Lemma \ref{supersolution_lemma} with $\gamma = \frac{p-1-n-2\delta}{p-1}$ and $v_1=v+\sqrt{\varepsilon}$. Note that $u \le Cv$ and $-u \le Cv$ on $(\{|x'| =\varepsilon/\mu\} \cup \{ |x'| = \mu/\kappa_2 \}) \cap \overline{\Omega}_{1}$ for some $\varepsilon$-independent constant $C$. By the comparison principle, we have
$|u| \le Cv$ in $\Omega_{\mu/\kappa_2}\setminus \Omega_{\varepsilon/\mu}$.
In particular, we have
\begin{align}\label{more-eps}
|u(x)| \le C(\varepsilon^{2} + |x'|^2)^{\frac{p-1-n-2\delta}{2(p-1)}}+C\varepsilon^{1/2} \quad \mbox{for}~~x \in \Omega_{\mu/\kappa_2}\setminus \Omega_{\varepsilon/\mu}.
\end{align}
Since $ \frac{p-1-n-2\delta}{2(p-1)}\in(0,1/2)$, by combining \eqref{less-eps} and \eqref{more-eps}, we obtain
\begin{equation*}
|u(x)| \le C(\varepsilon + |x'|^2)^{\frac{p-1-n-2\delta}{2(p-1)}} \quad \mbox{for}~~x \in \Omega_{\mu/\kappa_2}.
\end{equation*}
This implies that for any $x \in \Omega_{\mu/2\kappa_2}$ and $\varepsilon\in(0,\mu^2/(\kappa_2+\kappa_2^2))$,
\begin{equation}\label{oscillation_2d}
\underset{\Omega_{x,\eta}}{\osc}~u \le C(\varepsilon + |x'|^2)^{\frac{p-1-n-2\delta}{2(p-1)}},
\end{equation}
where $\eta = \frac{1}{4}(\varepsilon + |x'|^2)^{1/2}$, and  $C$ is a constant depending only on $n$, $p$, $\delta$, $\kappa_1$, $\kappa_2$, and the modulus of continuity for $D^2h_1(x')$ and $D^2 h_2(x')$ at $x'=0$.
Then \eqref{gradient_2d} follows from \eqref{oscillation_2d} and Theorem \ref{thm-1/2}.
\end{proof}

\section{A two dimensional example}\label{sec5}

 In this section, we provide an example showing that the estimates \eqref{gradient-1/2} and \eqref{gradient_2d} are close to optimal in 2D. In the following and throughout this section, we set our domain $\Omega = B_5 \subset \bR^2$, $D_1$ and $D_2$ to be the unit balls centered at $(0,1+\varepsilon/2)$ and $(0, -1-\varepsilon/2)$, respectively. That is,
\begin{equation}\label{Gamma_pm_2d_balls}
\Gamma_+ = \left\{ x_2 = \frac{\varepsilon}{2}+ 1 - \sqrt{1-x_1^2} \right\},~ \Gamma_- = \left\{ x_2 = -\frac{\varepsilon}{2}- 1 + \sqrt{1-x_1^2}\right\},~x_1 \in (-1,1).
\end{equation}
\begin{lemma}\label{subsolution_lemma}
Let $n=2$, $\Gamma_+$ and $\Gamma_-$ be as \eqref{Gamma_pm_2d_balls}, for any $\delta \in(0, 1/2)$, $\varepsilon \in(0,\delta/10)$, and $\gamma > \max\{ \frac{p-3+\delta}{p-1},0\}$, there exists a constant $r_0\in(0,1)$ depending only on $\delta$, such that the function
$$
w(x):= \Big[\Big(x_1^2 + (2-\delta)x_2^2\Big)^{\frac{\gamma}{2}} - (4 \sqrt{\varepsilon/\delta})^{\gamma} \Big]_+
$$
satisfies
\begin{equation}\label{def:super}
\left\{
\begin{aligned}
-\dv (|D w|^{p-2} D w) &\le 0 \quad \mbox{in }\Omega_{r_0},\\
\frac{\partial w}{\partial \nu} &\le 0 \quad \mbox{on } (\Gamma_+ \cup \Gamma_-) \cap \overline{\Omega}_{r_0}.\\
\end{aligned}
\right.
\end{equation}
\end{lemma}

\begin{proof}
We denote $R = R(x) = \Big(x_1^2 + (2-\delta)x_2^2\Big)^{\frac{1}{2}}$ and $v(x) = R(x)^\gamma$. Then
$$
D v = \gamma R^{\gamma - 2} (x_1, (2 - \delta)x_2).
$$
On $\Gamma_+$, the upward normal vector $\nu = (-x_1, 1 + \varepsilon/2 - x_2)$. Therefore,
\begin{align*}
\frac{\partial v}{\partial \nu} =& \gamma R^{\gamma - 2} \Big[ -x_1^2 + (2 - \delta)x_2 \Big( 1 + \frac{\varepsilon}{2} - x_2 \Big) \Big]\\
=& \gamma R^{\gamma - 2} \Big[ \Big(x_2 - 1 - \frac{\varepsilon}{2} \Big)^2 - 1 + (2 - \delta)x_2 \Big( 1 + \frac{\varepsilon}{2} - x_2 \Big) \Big]\\
=& \gamma R^{\gamma - 2} \Big[ -(1 - \delta) x_2^2 - \delta x_2 \Big( 1 + \frac{\varepsilon}{2} \Big) + \varepsilon + \frac{\varepsilon^2}{4} \Big].
\end{align*}
One can see that $\partial v/\partial \nu < 0$ if $x_2 > \varepsilon/\delta$. Since $x_2 = \frac{\varepsilon}{2}+ 1 - \sqrt{1-x_1^2}$,  $|x_1| > 2 \sqrt{\varepsilon/\delta}$ implies $x_2 > \varepsilon/\delta$. Note that $w = 0$ on $\Gamma_+$ when $|x_1| \le 2\sqrt{\varepsilon/\delta}$, therefore $\partial w/\partial \nu \le 0$ on $\Gamma_+$. The fact that $\partial w/\partial \nu \le 0$ on $\Gamma_-$ follows from a similar argument.

Next, following a similar computation as Lemma \ref{supersolution_lemma}, we have in the region $\{x_1^2 + (2-\delta)x_2^2 > 16\varepsilon/\delta \} \cap \Omega_1$,
\begin{align*}
&\dv (|D w|^{p-2} D w)|D w|^{4-p}\\
&=\gamma^2 R^{2\gamma - 4} \Big( x_1^2 + (2-\delta)^2x_2^2 \Big)\Big( (3-\delta)\gamma R^{\gamma-2} + \gamma(\gamma-2)R^{\gamma-4}(x_1^2 + (2-\delta)^2x_2^2)\Big)\\
&\quad + (p-2) \gamma^2 R^{2\gamma - 4} x_1^2 \Big( \gamma R^{\gamma-2} + \gamma(\gamma-2)R^{\gamma-4}x_1^2 \Big)\\
&\quad + (p-2) \gamma^2 R^{2\gamma - 4} (2-\delta)^2x_2^2 \Big( (2-\delta)\gamma R^{\gamma-2} + (2-\delta)^2\gamma(\gamma-2)R^{\gamma-4}x_2^2 \Big)\\
&\quad +  2(2+\delta)^2(p-2)\gamma^2 R^{2\gamma - 4}\gamma(\gamma-2)R^{\gamma-4}x_1^2 x_2^2 .
\end{align*}
Note that in $\{x_1^2 + (2-\delta)x_2^2 > 16\varepsilon/\delta \}\cap \Omega_1$, $|x_1| > 2\sqrt{\varepsilon/\delta}$. Therefore,
$$
|x_2|\le \frac{\varepsilon}{2}+ 1 - \sqrt{1-x_1^2} \le \frac{\delta}{8} x_1^2 + 1 - \sqrt{1-x_1^2} \le \Big(\frac{\delta}{8} + 1 \Big) x_1^2,
$$
and thus
$$
R^2 - (2 - \delta)\Big(\frac{\delta}{8} + 1 \Big)^2 R^4 \le x_1^2 \le R^2.
$$
Hence for small $R>0$, the leading term of $\dv (|D w|^{p-2} D w)|D w|^{4-p}$ is given by
$$
 \gamma^2 R^{2\gamma - 2} \Big((3-\delta)\gamma + \gamma(\gamma-2)\Big) R^{\gamma-2} +(p-2)\gamma^2 R^{2\gamma - 2}\gamma(\gamma-1)R^{\gamma-2}.
$$
If we set
$$
(3-\delta) + (\gamma-2) + (p-2)(\gamma-1) > 0 \quad \mbox{and} \quad \gamma > 0,
$$
which implies
$$
\gamma > \max \Big\{ \frac{p-3+\delta}{p-1}, 0 \Big\},
$$
then there exists an $r_0$ small, depending only on $\delta$, such that $\dv (|D w|^{p-2} D w) \ge 0$ in $\Omega_{r_0}$.
\end{proof}

\begin{proof}[Proof of Theorem \ref{thm:2d:2}]
We only need to prove the theorem for any $\delta \in (0,1/2)$ and $ \varepsilon \in(0, r_0^2\delta/64)$, where $r_0$ is the constant stated in Lemma \ref{subsolution_lemma}. By symmetry and the maximum principle, we have $u(0,x_2) = 0$ for $|x_2| < \varepsilon/2$ and $u(x)> 0$ when $x_1 > 0$. Let $w$ be the function defined in Lemma \ref{subsolution_lemma} with
$$
\gamma = \left\{
\begin{aligned}
&\delta &&\mbox{when}~~1 < p \le 3,\\
&\frac{p-3+2\delta}{p-1} &&\mbox{when}~~p > 3.
\end{aligned}
\right.
$$
By the comparison principle, there exists a positive constant $C$ depending only on $p$ and $\delta$, such that $u \ge \frac{1}{C} w$ in $\Omega_{r_0} \cap \{ x_1 > 0\}$. In particular, since $8\sqrt{\varepsilon/\delta} < r_0$, we have
$$
u(8\sqrt{\varepsilon/\delta}, 0) \ge \frac{1}{C} \varepsilon^{\frac{\gamma}{2}}.
$$
The desired lower bounds on $D u$ follow from the mean value theorem since $u(0) = 0$.
\end{proof}
\begin{remark}
When $\varepsilon=0$ and $p>5$, it can be shown that $w(x):=(x_1^2+2x_2^2)^{\gamma/2}$ is also a subsolution satisfying \eqref{def:super} for $\gamma=(p-3)/(p-1)$ and some absolute constant $r_0\in(0,1)$. Therefore, a similar argument as in the proof of Theorem \ref{thm:2d:2} gives $u(x_1,0)\geq c_0 \,x_1^{(p-3)/(p-1)}$ for some constant $c_0=c_0(p)>0$ and any $x_1\in (0,r_0)$.
Thus, for any $r\in(0,1)$, there exists $x_1\in(0,r)$ such that
$$D_1u(x_1,0)\geq c\,  x_1^{-2/(p-1)},$$
where $c>0$ is a constant depending only on $p$.
\end{remark}

\section{Bernstein type argument}\label{sec6}
In this section, we adapt the Bernstein type argument used in \cite{We} (see also \cite{CirSci2} and \cite{CirSci}) to prove improved gradient estimates for \eqref{main_problem_narrow} in high dimensions. As mentioned before, our proof also relies on the fact that for any $q\geq p$, $|Du|^q$ is a subsolution to the normalized $p$-Laplace equation, which was originally observed by Uhlenbeck \cite{MR474389}. In addition to \eqref{fg_0} and \eqref{fg_1}, we need to further assume that $h_1$ and $h_2$ are $C^{2,\text{Dini}}$ (so that $u$ is $C^2$ at the points where $D u \neq 0$, see \cite{DongLi}*{Theorem 2.4}), strictly convex and strictly concave respectively, satisfying \eqref{fg_convex}.

Let $\nu$ denote the normal vector on $\Gamma_{\pm}$, pointing upwards and downwards respectively. We have the following lemma.

\begin{lemma}
Let $\Gamma_+, \Gamma_-, h_1, h_2$ be as above, $s \ge 2$. If $u$ is twice differentiable and $D_\nu u = 0$ on $\Gamma_+ \cup \Gamma_-$, then at any point $x_0\in \Gamma_+\cup\Gamma_-$,
\begin{equation}\label{normal_derivative}
s\kappa_1|D u(x_0)|^s \le D_\nu |D u(x_0)|^s \le s \kappa_2 |D u(x_0)|^s.
\end{equation}
\end{lemma}
\begin{proof}
We only prove \eqref{normal_derivative} at $x_0 \in \Gamma_+$. By a rotation, we may assume that $x_0' = 0$ and $D_{x'}h_1(x_0') = 0$. The normal vector $\nu$ on $\Gamma_+$ is given by
\begin{equation}\label{normal_vector}
\nu = \frac{1}{\sqrt{1 + |D_{x'} h_1|^2}} (-D_1 h_1, \cdots, - D_{n-1} h_1 , 1).
\end{equation}
Then $D_\nu u = 0$ is equivalent to
$$
\sum_{j = 1}^{n-1} D_j u D_j h_1 - D_n u = 0.
$$
Applying $D_i$ to the equation above for $i = 1,\ldots, n-1$, we have at $x_0$,
\begin{equation}\label{eq1}
\sum_{j = 1}^{n-1} D_{j} u D_{ij} h_1 - D_{in} u = 0.
\end{equation}
By direct computation, at $x_0$,
\begin{align*}
D_\nu |D u|^s &= s |D u|^{s-2} \sum_{i=1}^n D_i u D_{in} u\\
&=  s |D u|^{s-2} \sum_{i=1}^{n-1} \sum_{j = 1}^{n-1} D_{i} u D_{j} u D_{ij} h_1,
\end{align*}
where in the second line, we used \eqref{eq1} and $D_n u(x_0) = 0$. Then \eqref{normal_derivative} follows from \eqref{fg_convex}.
\end{proof}

\begin{proof}[Proof of Theorem \ref{thm:bern}]
For convenience, we let $\gamma=2\beta\in[0,1)$.

{\em Case 1:} For $p \ge 2$, we consider the quantity
$$
F = Q^{\frac{p-p\gamma}{2}}|D u|^p,
$$
where
\begin{equation}\label{Q}
Q = \frac{\varepsilon}{\kappa_1} + |x'|^2 - \frac{5\kappa_2}{2(1-\gamma)\kappa_1} x_n^2.
\end{equation}
We will show by contradiction that $F$ does not achieve its maximum on $(\Gamma_+ \cup \Gamma_-) \cap \overline{\Omega}_{r_0}$ or in $\Omega_{r_0}$ for some suitable $r_0$ which is independent of $\varepsilon$. Therefore, $F$ can only achieve maximum on
$$
\{|x'| = r_0\} \cap \Omega_{1},
$$
and \eqref{gradient_-1/2_improved} follows.

If $F^{q/2}$ achieves it maximum at a point $x_0$, we may assume that $D u (x_0) \neq 0$. First we show that $x_0 \not\in \Gamma_+ \cap \overline{\Omega}_{r_0}$. A similar argument applies to $\Gamma_- \cap \overline{\Omega}_{r_0}$. On $\Gamma_+$, the normal vector $\nu$ is given by \eqref{normal_vector}. At $x_0$, by \eqref{normal_derivative} with $s = p$,
\begin{align*}
D_\nu F =& \frac{p - p\gamma}{2} Q^{\frac{p - p\gamma}{2} - 1} D_\nu Q |D u|^p +  Q^{\frac{p - p\gamma}{2}} D_\nu |D u|^p\\
\le& - \frac{(p-p\gamma) Q^{\frac{p - p\gamma}{2} - 1}}{\sqrt{1 + |D_{x'} h_1|^2}} \Big[ \sum_{j = 1}^{n-1} D_j h_1 x_j + \frac{5\kappa_2}{2(1-\gamma)\kappa_1} (\varepsilon/2 + h_1) \Big]|D u|^p\\
 &+ p\kappa_2 Q^{\frac{p - p\gamma}{2}} |D u|^p.
\end{align*}
We choose $r_0$ small enough such that
$$
\frac{-1}{\sqrt{1 + |D_{x'} h_1|^2}} \le \frac{-1}{\sqrt{1 + |\kappa_2 x'|^2}} \le -\frac{4}{5} \quad \mbox{for}~~|x'|< r_0.
$$
By \eqref{fg_0} and \eqref{fg_convex}, we have
$$
\sum_{j = 1}^{n-1} D_j h_1 x_j \ge \kappa_1 |x'|^2 \quad \mbox{and} \quad h_1 \ge \frac{1}{2}\kappa_1|x'|^2.
$$
Therefore,
\begin{align*}
D_\nu F \le &Q^{\frac{p - p\gamma}{2} - 1} |D u|^p \bigg( -\frac{4}{5}(p-p\gamma)\Big[\kappa_1|x'|^2 + \frac{5\kappa_2}{4(1-\gamma)\kappa_1} (\varepsilon + \kappa_1 |x'|^2) \Big]\\
&+ p\kappa_2 \Big( \frac{\varepsilon}{\kappa_1} + |x'|^2 - \frac{5\kappa_2}{2(1-\gamma)\kappa_1} x_n^2 \Big)  \bigg)\\
=& Q^{\frac{p - p\gamma}{2} - 1} |D u|^p \Big( -\frac{4}{5}(p-p\gamma)\kappa_1|x'|^2 - \frac{5p\kappa_2^2}{2(1-\gamma)\kappa_1}x_n^2 \Big) < 0.
\end{align*}
Hence $F$ does not achieve its maximum on $\Gamma_+ \cap \overline{\Omega}_{r_0}$. Next, we will show that $x_0 \not\in \Omega_{r_0}$ by assuming otherwise and showing that $a^{ij} D_{ij} F(x_0) > 0$, where
\begin{equation}\label{a_ij}
a^{ij}(x) = \delta_{ij} + (p-2)|D u|^{-2} D_i u D_j u
\end{equation}
is symmetric. Note that
\begin{equation}\label{a_elliptic}
|\xi|^2 \le a^{ij}\xi_i \xi_j \le (p-1) |\xi|^2, \quad \forall \xi \in \bR^n,
\end{equation}
and if $u$ is a solution of \eqref{main_problem_narrow}, then $a^{ij}D_{ij} u = 0$. Since $D u(x_0) \neq 0$. By the continuity of $D u$, $D u \neq 0$ in a neighborhood of $x_0$, and hence $a^{ij}$ is well defined in the neighborhood. By direction computations,
\begin{align}\label{3-1}
&a^{ij}D_{ij}F \notag\\
&= Q^{\frac{p-p\gamma}{2}} (a^{ij}D_{ij}|D u|^p) + |D u|^p (a^{ij}D_{ij}Q^{\frac{p-p\gamma}{2}}) + 2a^{ij}(D_{i}|D u|^p) (D_{j}Q^{\frac{p-p\gamma}{2}}).
\end{align}
Next, we estimate the three terms on the right-hand side above. First,
\begin{align}\label{3-2}
&a^{ij}D_{ij}Q^{\frac{p-p\gamma}{2}} \notag\\
&= a^{ij} \Big[\frac{p-p\gamma}{2}Q^{\frac{p - p\gamma}{2} - 1} D_{ij} Q + \frac{p-p\gamma}{2} \Big(\frac{p-p\gamma}{2} -1 \Big)Q^{\frac{p - p\gamma}{2} - 2} D_i Q D_j Q \Big].
\end{align}
Since $a^{ij}D_{ij} u = 0$, applying $D_k$ gives
$$
a^{ij} D_{ijk} u + D_k a^{ij} D_{ij} u = 0.
$$
Since
$$
D_k a^{ij} = (p-2) \left[ \frac{D_{ik}u D_j u + D_i u D_{jk}u}{|D u|^2} - 2 \frac{D_i u D_j u D_{kl} u D_l u}{|D u|^4} \right],
$$
we have
\begin{align}\label{identity_1}
&a^{ij} D_{ijk} u D_k u = - D_k a^{ij} D_{ij} u D_k u\notag\\
&= -2(p-2)|D |D u||^2 + 2(p-2)|D u|^{-4} |\Delta_\infty u|^2,
\end{align}
where $\Delta_\infty u := D_i u D_j u D_{ij} u$.
By \eqref{a_ij} and \eqref{identity_1}, we have
\begin{align}\label{3-3}
a^{ij}D_{ij}|D u|^p =& pa^{ij} \Big[ (p-2)|D u|^{p-4} D_{ik}u D_k u D_{jl}u D_l u \nonumber \\
&+ |D u|^{p-2} D_{ik}u D_{jk}u + |D u|^{p-2} D_k u D_{ikj} u  \Big] \nonumber\\
=& p|D u|^{p-4} \Big[ (p-2) |D u|^2 |D |D u||^2 + (p-2)^2 |D u|^{-2} |\Delta_\infty u|^2 \nonumber\\
&+ |D u|^2 |D^2 u|^2 + (p-2) |D u|^2 |D |D u||^2 - 2(p-2) |D u|^2 |D |D u||^2 \nonumber\\
&+ 2(p-2) |D u|^{-2} |\Delta_\infty u|^2 \Big] \nonumber\\
=& p|D u|^{p-4} \Big[ p(p-2) |D u|^{-2} |\Delta_\infty u|^2 +  |D u|^2 |D^2 u|^2 \Big].
\end{align}
Note that at the point $x_0$, for any $i= 1,2,\ldots,n$,
\begin{equation}\label{extreme_point}
0 = D_i F = (D_i Q^{\frac{p - p\gamma}{2}})|D u|^p + Q^{\frac{p - p\gamma}{2}} (D_i |D u|^p).
\end{equation}
We split the last term on the right-hand side of \eqref{3-1} into
\begin{align*}
\frac{2p-1}{p} a^{ij}(D_{i}|D u|^p) (D_{j}Q^{\frac{p-p\gamma}{2}}) + \frac{1}{p} a^{ij}(D_{i}|D u|^p) (D_{j}Q^{\frac{p-p\gamma}{2}}).
\end{align*}
For the first term on the right-hand side above, we use \eqref{extreme_point} to substitute $D_{i}|D u|^p$, and for the second term, we substitute $D_{j}Q^{\frac{p-p\gamma}{2}}$. Then,
\begin{align}\label{3-4}
&2a^{ij}(D_{i}|D u|^p) (D_{j}Q^{\frac{p-p\gamma}{2}}) \notag\\
&=  -\frac{2p-1}{p} Q^{-\frac{p-p\gamma}{2}}|D u|^p a^{ij} (D_{i}Q^{\frac{p-p\gamma}{2}})(D_{j}Q^{\frac{p-p\gamma}{2}}) \nonumber\\
&\quad - \frac{1}{p} Q^{\frac{p-p\gamma}{2}}|D u|^{-p} a^{ij}(D_{i}|D u|^p) (D_{j}|D u|^p)\nonumber \\
&= -\frac{2p-1}{p} \frac{(p-p\gamma)^2}{4} Q^{\frac{p-p\gamma}{2} -2}|D u|^p  a^{ij} D_{i}Q D_{j}Q \nonumber\\
&\quad - p Q^{\frac{p-p\gamma}{2}}|D u|^{p-4} a^{ij} D_{ik}u D_k u D_{jl}u D_l u \nonumber\\
&= -\frac{2p-1}{p} \frac{(p-p\gamma)^2}{4} Q^{\frac{p-p\gamma}{2} -2}|D u|^p  a^{ij} D_{i}Q D_{j}Q \nonumber\\
&\quad - p Q^{\frac{p-p\gamma}{2}}|D u|^{p-4} (|D u|^2 |D |D u||^2 + (p-2) |D u|^{-2} |\Delta_\infty u|^2),
\end{align}
where we used \eqref{a_ij} in the last equality. Therefore, by \eqref{3-1}, \eqref{3-2}, \eqref{3-3}, and \eqref{3-4},
\begin{align*}
a^{ij}D_{ij}F =&  p Q^{\frac{p-p\gamma}{2}}|D u|^{p-4} \Big[ (p-1)(p-2) |D u|^{-2} |\Delta_\infty u|^2 +  |D u|^2 |D^2 u|^2\\
&-  |D u|^2 |D |D u||^2\Big] + \frac{p-p\gamma}{2}Q^{\frac{p - p\gamma}{2} - 1}|D u|^p  a^{ij}  D_{ij} Q \\
&- \Big[ \frac{2p-1}{p} \frac{(p-p\gamma)^2}{4} - \frac{p-p\gamma}{2} \Big(\frac{p-p\gamma}{2} -1 \Big) \Big] Q^{\frac{p-p\gamma}{2} -2}|D u|^p  a^{ij} D_{i}Q D_{j}Q.
\end{align*}
Note that
$$
|D u|^2 |D |D u||^2 \le |D u|^2 |D^2 u|^2.
$$
It remains to show
\begin{equation}\label{remaining}
 Q  a^{ij}  D_{ij} Q > \Big[ \frac{2p-1}{p} \frac{p-p\gamma}{2} - \Big(\frac{p-p\gamma}{2} -1 \Big) \Big] a^{ij} D_{i}Q D_{j}Q.
\end{equation}
Recall that $Q$ is given in \eqref{Q}. Then
\begin{equation}\label{nabla_Q}
D Q = \Big(2x_1, \ldots, 2x_{n-1}, -\frac{5\kappa_2}{(1-\gamma)\kappa_1} x_n \Big),
\end{equation}
and
\begin{align*}
a^{ij}  D_{ij} Q =& 2n-2 - \frac{5\kappa_2}{(1-\gamma)\kappa_1}\\
&+ (p-2)|D u|^{-2}\Big(2|D_{x'} u|^2 - \frac{5\kappa_2}{(1-\gamma)\kappa_1} |D_n u|^2 \Big)\\
\ge & 2n-2 - (p-1) \frac{5\kappa_2}{(1-\gamma)\kappa_1}.
\end{align*}
By \eqref{a_elliptic} and shrinking $r_0$ if necessary, we have
\begin{align*}
a^{ij} D_{i}Q D_{j}Q \le 4(p-1)\Big(|x'|^2 +  \frac{25\kappa_2^2}{4(1-\gamma)^2\kappa_1^2} x_n^2 \Big)
< 5(p-1) Q.
\end{align*}
In order to show \eqref{remaining}, we only require
$$
2n-2 - (p-1) \frac{5\kappa_2}{(1-\gamma)\kappa_1} \ge 5(p-1)\Big[ \frac{2p-1}{p} \frac{p-p\gamma}{2} - \Big(\frac{p-p\gamma}{2} -1 \Big) \Big],
$$
which is equivalent to \eqref{np_relation_1}  since $\gamma=2\beta$. This concludes the proof for the case when $p \ge 2$.

{\em Case 2:} For $ p \in(1, 2)$, we consider the quantity
$$
G = Q^{1-\gamma}|D u|^2,
$$
where $Q$ is given in \eqref{Q}. From the computation of $D_\nu F$ with $p=2$, one can see that $G$ does not attain its maximum on $(\Gamma_+ \cup \Gamma_-) \cap \overline{\Omega}_{r_0}$. Next, we assume that $G$ achieves its maximum at $x_0 \in \Omega_{r_0}$. By the computations as in \eqref{3-1} and \eqref{3-2} with $p=2$, we have
\begin{align}\label{4-1}
a^{ij}D_{ij}G = Q^{1-\gamma} (a^{ij}D_{ij}|D u|^2) + |D u|^2 (a^{ij}D_{ij}Q^{1-\gamma}) + 2a^{ij}(D_{i}|D u|^2) (D_{j}Q^{1-\gamma}),
\end{align}
where $a^{ij}$ is given in \eqref{a_ij} with
\begin{equation}\label{a_elliptic_3}
(p-1) |\xi|^2 \le a^{ij}\xi_i \xi_j \le |\xi|^2, \quad \forall \xi \in \bR^n.
\end{equation}
Next, we estimate the three terms on the right-hand side of \eqref{4-1}. First,
\begin{align}\label{4-2}
a^{ij}D_{ij}Q^{1-\gamma} = a^{ij} \Big[(1-\gamma)Q^{-\gamma} D_{ij} Q - \gamma(1-\gamma) Q^{-1-\gamma} D_i Q D_j Q \Big].
\end{align}
By \eqref{a_ij} and \eqref{identity_1}, we have
\begin{align}\label{4-3}
a^{ij}D_{ij}|D u|^2 =& a^{ij} \Big[2 D_{ijk} u D_k u + 2D_{jk} u D_{ik} u \Big] \nonumber\\
=& 4(2-p)|D |D u||^2 - 4(2-p)|D u|^{-4} |\Delta_\infty u|^2 \nonumber\\
&+ 2|D^2 u|^2 - 2(2-p)|D |D u||^2 \nonumber\\
=& 2(2-p)|D |D u||^2 - 4(2-p)|D u|^{-4} |\Delta_\infty u|^2 + 2|D^2 u|^2.
\end{align}
Note that at the point $x_0$, for any $i= 1,2,\ldots,n$,
\begin{equation}\label{extreme_point_2}
0 = D_i G = (D_i Q^{1-\gamma})|D u|^2 + Q^{1-\gamma} (D_i |D u|^2).
\end{equation}
As before, we split the last term on the right-hand side of \eqref{4-1} as
\begin{align*}
\frac{1}{2} a^{ij}(D_{i}|D u|^2) (D_{j}Q^{1-\gamma}) + \frac{3}{2} a^{ij}(D_{i}|D u|^2) (D_{j}Q^{1-\gamma}).
\end{align*}
For the first term on the right-hand side, we use \eqref{extreme_point_2} to substitute $D_{j}Q^{1-\gamma}$, and for the second term, we substitute $D_{i}|D u|^2$. Then by \eqref{a_ij},
\begin{align}\label{4-4}
2a^{ij}(D_{i}|D u|^2) (D_{j}Q^{1-\gamma}) =& -\frac{1}{2}a^{ij}(D_{i}|D u|^2)(D_{j}|D u|^2) Q^{1-\gamma}|D u|^{-2} \nonumber \\
&- \frac{3}{2}a^{ij}(D_{i}Q^{1-\gamma})(D_{j}Q^{1-\gamma}) |D u|^2 Q^{-1+\gamma}\nonumber \\
=& -2Q^{1-\gamma}|D u|^{-2}a^{ij} D_{ik} u D_k u D_{jl} u D_l u  \nonumber \\
&- \frac{3}{2} (1-\gamma)^2  |D u|^2 Q^{-1-\gamma} a^{ij} D_i Q D_j Q \nonumber \\
=& -2Q^{1-\gamma} |D |D u||^2 + 2 (2-p) Q^{1-\gamma}|D u|^{-4} |\Delta_\infty u|^2 \nonumber \\
&- \frac{3}{2}(1-\gamma)^2  |D u|^2 Q^{-1-\gamma} a^{ij} D_i Q D_j Q.
\end{align}
Therefore, by \eqref{4-1}, \eqref{4-2}, \eqref{4-3}, and \eqref{4-4},
\begin{align*}
a^{ij}D_{ij}G =& Q^{1-\gamma} \Big[ (2(2-p)-2) |D |D u||^2 - 2(2-p)|D u|^{-4} |\Delta_\infty u|^2 + 2|D^2 u|^2 \Big]\\
&+  (1-\gamma)Q^{-\gamma} |D u|^2 a^{ij}D_{ij} Q\\
&- \Big[ \gamma(1-\gamma) + \frac{3}{2}(1-\gamma)^2 \Big] |D u|^2 Q^{-1-\gamma} a^{ij} D_i Q D_j Q.
\end{align*}
Since
$$
|D u|^{-4} |\Delta_\infty u|^2 \le |D |D u||^2 \le |D^2 u|^2,
$$
it remains to show
\begin{equation}\label{remaining_2}
 Q  a^{ij}  D_{ij} Q > \Big[ \gamma + \frac{3}{2}(1-\gamma) \Big]  a^{ij} D_i Q D_j Q.
\end{equation}
By \eqref{nabla_Q}, we have
\begin{align*}
a^{ij}  D_{ij} Q =& 2n-2 - \frac{5\kappa_2}{(1-\gamma)\kappa_1}\\
&- (2-p)|D u|^{-2}\Big(2|D_{x'} u|^2 - \frac{5\kappa_2}{(1-\gamma)\kappa_1} |D_n u|^2 \Big)\\
\ge & 2n-2 - \frac{5\kappa_2}{(1-\gamma)\kappa_1} -2(2-p).
\end{align*}
By \eqref{a_elliptic_3} and shrinking $r_0$ if necessary, we have
\begin{align*}
a^{ij} D_{i}Q D_{j}Q \le 4\Big(|x'|^2 +  \frac{25\kappa_2^2}{4(1-\gamma)^2\kappa_1^2} x_n^2 \Big)
< 5Q.
\end{align*}
In order to show \eqref{remaining_2}, we only require
$$
2n-2 - \frac{5\kappa_2}{(1-\gamma)\kappa_1} -2(2-p) \ge 5 \Big[ \gamma + \frac{3}{2}(1-\gamma) \Big],
$$
which is equivalent to \eqref{np_relation_2} since $\gamma=2\beta$. This concludes the proof for the case when $ p \in(1,2)$.
\end{proof}


\appendix

\section{}
In the appendix, we provide an alternative proof of the gradient estimates of order $\varepsilon^{-1/2}$ using a Bernstein type argument. This proof also requires the assumptions that $h_1$ and $h_2$ are $C^{2,\text{Dini}}$ functions and satisfy \eqref{fg_convex} for some $\kappa_1,\kappa_2 > 0$, in addition to \eqref{fg_0} imposed in Theorem \ref{thm-1/2}.

\begin{theorem}
Let $h_1$, $h_2$ be $C^{2,\text{Dini}}$ functions satisfying \eqref{fg_convex}, $p > 1$, $n \ge 2$, $\varepsilon\in (0,1)$, and $u \in W^{1,p}(\Omega_{1})$ be a solution of \eqref{main_problem_narrow}. Then there exists a positive constant $C$ depending only on $n$, $p$, $\kappa_1$, and $\kappa_2$, such that
\begin{equation}
                    \label{gradient_-1/2}
|D u(x)| \le C \|u\|_{L^\infty(\Omega_{1})} (\varepsilon + |x'|^2)^{-1/2} \quad \mbox{for}~~x \in \Omega_{1/2}.
\end{equation}
\end{theorem}

\begin{proof}
Without loss of generality, we may assume $\kappa_1\in(0,1]$ and $\kappa_2>1$.
The case $p=2$ has been shown in \cites{BLY2, We}. It remains to show the cases when $p > 2$ and $ p \in(1, 2)$.

{\em Case 1:} For $p > 2$, we consider the quantity $F^{q/2}$, where
\begin{equation}\label{FQ}
F = Q|D u|^2 + A u^2, \quad Q = \varepsilon + |x'|^2 - 4\kappa_1^{-1} \kappa_2^2 x_n^2,
\end{equation}
$q \ge 2$ and $A$ are some positive $\varepsilon$-independent constants to be determined later. Let $$
S_A : = \{ Q|D u|^2 > 100 Au^2 \}.
$$
We will show that $F^{q/2}$ does not achieve its maximum on $(\Gamma_+ \cup \Gamma_-) \cap \overline{\Omega}_{r_0}\cap S_A$ or in $\Omega_{r_0} \cap S_A$ for some suitable $q$, $A$, and $r_0$. Therefore, $F^{q/2}$ can only achieve its maximum in
$$
\Omega_{r_0} \cap \{ Q|D u|^2 \le 100 Au^2 \},
$$
or on
$$
\{|x'| = r_0\} \cap \Omega_{1},
$$
so \eqref{gradient_-1/2} follows from either case.

First we show that $F^{q/2}$ does not achieve its maximum on $\Gamma_+ \cap \overline{\Omega}_{r_0}\cap S_A$. A similar argument applies to $\Gamma_- \cap \overline{\Omega}_{r_0}\cap S_A$. On $\Gamma_+$, the normal vector $\nu$ is given by \eqref{normal_vector}. Then
\begin{align*}
&D_\nu F^{q/2} = \frac{q}{2} F^{q/2-1} (D_\nu Q |D u|^2 + Q D_\nu |D u|^2)\\
&= \frac{q}{2} F^{q/2-1} \Bigg( \frac{-2}{\sqrt{1 + |D_{x'} h_1|^2}} \Big[ \sum_{j = 1}^{n-1} D_j h_1 x_j + 4\frac{\kappa_2^2}{\kappa_1} (\varepsilon/2 + h_1) \Big]|D u|^2 + Q D_\nu |D u|^2) \Bigg).
\end{align*}
We choose $r_0$ small enough such that
$$
\frac{-2}{\sqrt{1 + |D_{x'} h_1|^2}} \le \frac{-2}{\sqrt{1 + |\kappa_2 x'|^2}} \le -1 \quad \mbox{for}~~|x'|< r_0.
$$
By \eqref{fg_0} and \eqref{fg_convex}, we have
$$
\sum_{j = 1}^{n-1} D_j h_1 x_j \ge \kappa_1 |x'|^2 \quad \mbox{and} \quad h_1 \ge \frac{1}{2}\kappa_1|x'|^2.
$$
Therefore, by \eqref{normal_derivative} with $s=2$, we have
\begin{align*}
&D_\nu F^{q/2} \\
& \le  \frac{q}{2} F^{q/2-1} |D u|^2 \Big[ -(\kappa_1|x'|^2 + 2 \frac{\kappa_2^2}{\kappa_1}\varepsilon + 2\kappa_2^2|x'|^2) + 2\kappa_2(\varepsilon + |x'|^2 - 4 \frac{\kappa_2^2}{\kappa_1} x_n^2)\Big]\\
&= - \frac{q}{2} F^{q/2-1} |D u|^2 \Big[ 2\kappa_2 (\frac{\kappa_2}{\kappa_1}-1) \varepsilon + 2\kappa_2(\kappa_2-1)|x'|^2 + \kappa_1|x'|^2 + 8 \frac{\kappa_2^2}{\kappa_1} x_n^2 \Big]\\
& < 0.
\end{align*}
Hence $F^{q/2}$ does not achieve its maximum on $\Gamma_+ \cap \overline{\Omega}_{r_0}\cap S_A$. Next, we will show $F^{q/2}$ does not achieve its maximum in $\Omega_{r_0} \cap S_A$ by proving that $a^{ij} D_{ij} F^{q/2} > 0$, where $a^{ij}$ is given in \eqref{a_ij}. By direct computations, we have,
\begin{align*}
a^{ij} D_{ij} F^{q/2} = \frac{q}{2} F^{q/2 - 1} a^{ij} D_{ij} F + \frac{q}{2}\Big( \frac{q}{2}-1 \Big) F^{q/2-2} a^{ij} D_i F D_j F,
\end{align*}
\begin{align*}
D_i F = D_i Q |D u|^2 + 2Q D_{ik}u D_k u + 2Au D_i u,
\end{align*}
and
\begin{align*}
D_{ij} F =& D_{ij}Q |D u|^2 + 2 D_i Q D_k u D_{jk} u +2 D_j Q D_k u D_{ik} u \\
&+ 2Q(D_{ik}u D_{jk} u + D_k u D_{ijk} u) + 2A(u D_{ij}u + D_i u D_j u).
\end{align*}
Then by \eqref{a_ij} and because $a^{ij}D_{ij} u = 0$,
\begin{align*}
&a^{ij} D_{ij} F \\
&=  a^{ij} D_{ij} Q |D u|^2 + 4(p-2)|D u|^{-2} D_i u D_i Q \Delta_\infty u + 4 D_i Q D_k u D_{ik} u + 2Q |D^2 u|^2\\
&\quad + 2(p-2)Q |D |D u||^2 + 2Q  a^{ij} D_{ijk} u D_k u + 2A a^{ij} D_i u D_j u,
\end{align*}
where $\Delta_\infty u := D_i u D_j u D_{ij} u$. By \eqref{identity_1},
\begin{align*}
&a^{ij} D_{ij} F \\
&= a^{ij} D_{ij} Q |D u|^2 + 4(p-2)|D u|^{-2} D_i u D_i Q \Delta_\infty u + 4 D_i Q D_k u D_{ik} u + 2Q |D^2 u|^2\\
&\quad - 2(p-2)Q |D |D u||^2  + 4 (p-2) Q|D u|^{-4} |\Delta_\infty u|^2 + 2A a^{ij} D_i u D_j u.
\end{align*}
By another direction computation, we have
\begin{align*}
&a^{ij} D_i F D_j F \\
&= a^{ij} D_i Q D_j Q |D u|^4 + 4Q^2 a^{ij} D_{ik}u D_ku D_{jl}u D_l u  + 4A^2 u^2 a^{ij} D_iu D_ju\\
&\quad + 4Q|D u|^2 a^{ij} D_i Q D_{jl} u D_lu + 4Au |D u|^2a^{ij} D_i Q D_j u + 8AQ u a^{ij} D_{ik} u D_k u D_ju\\
&= a^{ij} D_i Q D_j Q |D u|^4 + 4(p-2)Q^2 |D u|^{-2} |\Delta_\infty u|^2 + 4Q^2 |D u|^2 |D |D u||^2\\
&\quad + 4(p-1)A^2 u^2|D u|^2 + 4(p-2) Q D_i u D_i Q \Delta_\infty u +4Q D_i Q D_{ik}u D_k u |D u|^2\\
&\quad + 4 A(p-1) |D u|^2 u D_i u D_i Q + 8(p-1)AQu \Delta_\infty u.
\end{align*}
Therefore,
\begin{align}\label{1-1}
&a^{ij} D_{ij} F^{q/2}\notag\\
&= \frac{q}{2} F^{q/2 -2} \big[ a^{ij} D_{ij} Q |D u|^2 F +2Q|D |D u||^2 \big( (q-2)Q|D u|^2 - (p-2)F \big) \nonumber\\
&\quad + 4(p-2) D_i u D_i Q \Delta_\infty u |D u|^{-2} (F + (q-2)Q|D u|^2/2) + 2QF|D^2 u|^2\nonumber\\
&\quad + 4 D_i Q D_k u D_{ik}u (F + (q-2)Q|D u|^2/2)\nonumber\\
&\quad + 4(p-2)Q|D u|^{-4} (F + (q-2)Q|D u|^2/2) |\Delta_\infty u|^2 \nonumber\\
&\quad  + 2AFa^{ij} D_i u D_j u + (q-2)a^{ij}|Du|^4 D_i Q D_j Q/2 + 2(p-1)(q-2) A^2 u^2 |D u|^2 \nonumber\\
&\quad + 2A(p-1)(q-2)|D u|^2 u D_i u D_i Q + 4(p-1)(q-2)AQu \Delta_\infty u \big].
\end{align}
Note that in $S_A$,
\begin{equation}\label{S_A}
Q|D u|^2 \le F \le \frac{101}{100} Q |D u|^2.
\end{equation}
By \eqref{a_elliptic},
$$
(q-2)a^{ij}D_i Q D_j Q/2 \ge (q-2)|D Q|^2 /2 \ge 0 \quad \mbox{for}~~q \ge 2,
$$
and
\begin{align}\label{1-2}
&a^{ij} D_{ij} Q |D u|^2 F + 2AFa^{ij} D_i u D_j u\nonumber\\
&= [2(n-1) - 8\kappa_1^{-1}\kappa_2^2 + (p-2)|D u|^{-2} (2|D_{x'} u|^2 - 8\kappa_1^{-1}\kappa_2^2 |D_n u|^2)] |D u|^2 F\nonumber\\
&\quad + 2AFa^{ij} D_i u D_j u\nonumber\\
&\ge \Big( 2(n-1) - 8\kappa_1^{-1}\kappa_2^2(p-1) + 2A \Big) |D u|^2 F\nonumber\\
&\ge \Big( 2(n-1) - 8\kappa_1^{-1}\kappa_2^2(p-1) + 2A \Big) |D u|^4 Q.
\end{align}
We choose $q = \frac{101}{100}(p-2) + 2 > p$, so that
\begin{align*}
&2Q|D |D u||^2 \big( (q-2)Q|D u|^2 - (p-2)F \big)\\
&\ge 2Q|D |D u||^2 \big( (q-2) - \frac{101}{100}(p-2) \big)Q|D u|^2 = 0.
\end{align*}
It remains to control
\begin{align*}
&4(p-2) D_i u D_i Q \Delta_\infty u |D u|^{-2} (F + (q-2)Q|D u|^2/2)\\
&+4 D_i Q D_k u D_{ik}u (F + (q-2)Q|D u|^2/2)\\
&+ 2A(p-1)(q-2)|D u|^2 u D_i u D_i Q + 4(p-1)(q-2)AQu \Delta_\infty u\\
&=: I + II + III + IV.
\end{align*}
Since $p > 2$, we have
$$
\frac{2(p-2)q}{(p-1)(q-2)} = \frac{2 \Big[ \frac{101}{100}(p-2) + 2\Big]}{\frac{101}{100}(p-1)} > 2.
$$ Fix a constant $B \in (2, \frac{2(p-2)q}{(p-1)(q-2)})$. We shrink $r_0$ if necessary so that $|D Q|^2 \le 8 Q$. By Young's inequality and \eqref{S_A},
\begin{align}\label{1-3}
|I| \le & \left( \frac{p-2}{2} - \frac{B(p-1)(q-2)}{4q} \right)|D u|^{-4} |D Q|^2 |\Delta_\infty u|^2 (F + (q-2)Q|D u|^2/2)\nonumber\\
& + C(p) |D u|^2 (F + (q-2)Q|D u|^2/2) \nonumber\\
\le &\left( 4(p-2) - \frac{2B(p-1)(q-2)}{q} \right) Q |D u|^{-4} |\Delta_\infty u|^2 (F + (q-2)Q|D u|^2/2)\nonumber\\
&+ C(p) |D u|^4Q,
\end{align}
where $C(p)$ is some positive constant depending on $p$. By Young's inequality and \eqref{S_A},
\begin{align}\label{1-4}
|II| &\le 4 \Big(\frac{101}{100} + \frac{q-2}{2}\Big) |D Q| |D^2 u| |D u|^3 Q\nonumber\\
& \le 4 \Big(\frac{101}{100} + \frac{q-2}{2}\Big)^2 |D Q|^2 |D u|^4 + Q^2 |D^2 u|^2 |D u|^2 \nonumber\\
& \le 32 \Big(\frac{101}{100} + \frac{q-2}{2}\Big)^2   |D u|^4 Q+ QF |D^2 u|^2,
\end{align}
\begin{align}\label{1-5}
|III| &\le 2(p-1)(q-2) A  |D u|^3|u| |D Q|\nonumber\\
&\le \left( 2 - \frac{4}{B} \right) (p-1)(q-2) A^2 u^2 |D u|^2 + \frac{4B}{B-2}(p-1)(q-2) |D u|^4 Q,
\end{align}
and
\begin{align}\label{1-6}
|IV| \le& \frac{4}{B}(p-1)(q-2) A^2 u^2 |D u|^2 + B(p-1)(q-2)Q^2 |D u|^{-2} |\Delta_\infty u|^2 \nonumber\\
\le& \frac{4}{B}(p-1)(q-2) A^2 u^2 |D u|^2 \nonumber\\
&+ \frac{2B(p-1)(q-2)}{q} Q |D u|^{-4} |\Delta_\infty u|^2 (F + (q-2)Q|D u|^2/2).
\end{align}
Now we choose $A$ large such that
$$
2(n-1) - 8\kappa_1^{-1}\kappa_2^2(p-1) + 2A - C(p)- 32 \Big(\frac{101}{100} + \frac{q-2}{2}\Big)^2 - \frac{4B}{B-2}(p-1)(q-2) > 0.
$$
Then by \eqref{1-1}, \eqref{1-2}, \eqref{1-3}, \eqref{1-4}, \eqref{1-5}, and \eqref{1-6}, $a^{ij} D_{ij} F^{q/2} > 0$ in $\Omega_{r_0} \cap S_A$, and hence $F^{q/2}$ does not achieve its maximum in $\Omega_{r_0} \cap S_A$. This concludes the proof for the case when $p > 2$.

{\em Case 2:} For $p \in(1,2)$, we consider the quantify $F$ given in \eqref{FQ}. A similar argument as above shows that $F$ does not achieve maximum on $(\Gamma_+ \cup \Gamma_-) \cap \overline{\Omega}_{r_0}\cap S_A$. In $\Omega_{r_0} \cap S_A$, we compute
\begin{align}\label{2-1}
&a^{ij} D_{ij} F\notag \\
&= a^{ij} D_{ij} Q |D u|^2 - 4(2-p)|D u|^{-2} D_i u D_i Q \Delta_\infty u + 4 D_i Q D_k u D_{ik} u + 2Q |D^2 u|^2\nonumber\\
&\quad + 2(2-p)Q |D |D u||^2  - 4 (2-p) Q|D u|^{-4} |\Delta_\infty u|^2 + 2A a^{ij} D_i u D_j u,
\end{align}
where $a^{ij}$ is given in \eqref{a_ij} satisfying \eqref{a_elliptic_3}.
By a direct computation and \eqref{a_elliptic_3}, we have
\begin{align}\label{2-2}
&a^{ij} D_{ij} Q |D u|^2 + 2Aa^{ij} D_i u D_j u\nonumber\\
&= [2(n-1) - 8\kappa_1^{-1}\kappa_2^2 + (p-2)|D u|^{-2} (2|D_{x'} u|^2 - 8\kappa_1^{-1}\kappa_2^2 |D_n u|^2)] |D u|^2 \nonumber\\
&\quad + 2Aa^{ij} D_i u D_j u\nonumber\\
&\ge \Big( 2(n-1) - 8\kappa_1^{-1}\kappa_2^2 -2(2-p) + 2A(p-1) \Big) |D u|^2.
\end{align}
Note that
$$
|D u|^{-4} |\Delta_\infty u|^2 \le |D |D u||^2 \le |D^2 u|^2.
$$
Therefore,
\begin{align}\label{2-3}
& 2Q |D^2 u|^2 + 2(2-p)Q |D |D u||^2  - 4 (2-p) Q|D u|^{-4} |\Delta_\infty u|^2\notag\\
&\ge  2(p-1)Q |D^2 u|^2.
\end{align}
It remains to control
$$
- 4(2-p)|D u|^{-2} D_i u D_i Q \Delta_\infty u + 4 D_i Q D_k u D_{ik} u =: I + II.
$$
By Young's inequality and $|D Q|^2 \le 8Q$ for small $r_0$, we have
\begin{align}\label{2-4}
|I| \le& 4(2-p) |D u|^{-1} |D Q| |\Delta_\infty u|\nonumber\\
\le & \frac{p-1}{16} |D Q|^2 |D u|^{-4} |\Delta_\infty u|^2 + \frac{64(2-p)^2}{p-1} |D u|^2\nonumber\\
\le & \frac{p-1}{2} Q|D u|^{-4} |\Delta_\infty u|^2 + \frac{64(2-p)^2}{p-1} |D u|^2,
\end{align}
and
\begin{align}\label{2-5}
|II| \le& 4|D Q| |D u| |D^2 u|\nonumber\\
\le & \frac{p-1}{16} |D Q|^2|D^2 u|^2 + \frac{64}{p-1}|D u|^2\nonumber\\
\le & \frac{p-1}{2} Q|D^2 u|^2 + \frac{64}{p-1}|D u|^2.
\end{align}
Now we choose $A$ large such that
$$
2(n-1) - 8\kappa_1^{-1}\kappa_2^2 -2(2-p) + 2A(p-1) - \frac{64(2-p)^2}{p-1} - \frac{64}{p-1} > 0.
$$
Then by \eqref{2-1}, \eqref{2-2}, \eqref{2-3}, \eqref{2-4}, and \eqref{2-5}, $a^{ij} D_{ij} F > 0$ in $\Omega_{r_0} \cap S_A$, and hence $F$ does not achieve its maximum in $\Omega_{r_0} \cap S_A$. This concludes the proof for the case when $p \in(1, 2)$.
\end{proof}

\bibliographystyle{amsplain}

\end{document}